\documentclass{amsart}

\usepackage[a-1b]{pdfx}   
\makeatletter \AtBeginDocument{\let\mathaccentV\AMS@mathaccentV} \makeatother

\usepackage{geometry}
\usepackage[parfill]{parskip} 

\usepackage[dvipsnames]{xcolor}
\usepackage{enumerate}
\usepackage{fancyvrb}

\usepackage{amssymb}
\usepackage{amsmath}
\usepackage{amsthm}

\numberwithin{equation}{section}
\usepackage[nameinlink]{cleveref}
\usepackage{cite}

\usepackage[acronym]{glossaries}
\glsdisablehyper
\newacronym{mhd}{MHD}{Magneto-Hydrodynamic}
\newacronym{ode}{ODE}{ordinary differential equation}

\usepackage{graphicx}
\usepackage{caption}
\usepackage{tikz}
\usetikzlibrary{patterns.meta}

\usepackage{keytheorems}
\newkeytheorem{theorem}[parent=section, style=plain]
\newkeytheorem{lemma, claim, corollary, proposition, conjecture, definition}[sibling=theorem, style=plain]

\renewkeytheoremstyle{remark}{
    inherit-style=remark,
    qed={$\triangle$},
}
\newkeytheorem{remark, example}[style=remark,sibling=theorem]

\usepackage{mathtools}
\DeclarePairedDelimiter\p{(}{)}

\DeclarePairedDelimiter\abs{\lvert}{\rvert}
\DeclarePairedDelimiter\norm{\lVert}{\rVert}
\DeclarePairedDelimiter\set{\{}{\}}

\newcommand{\bb}[1]{\mathbb{#1}}
\def\R{\bb{R}}

\newcommand{\cc}[1]{\mathcal{#1}}

\usepackage{bbm}
\newcommand{\bbm}[1]{\mathbbm{#1}}

\def\d{\,\mathrm{d}}
\DeclareMathOperator{\proj}{proj}

\DeclareMathOperator{\tdiv}{div}
\DeclareMathOperator{\tcurl}{curl}

\DeclareMathOperator{\supp}{supp}

\newcommand{\radial}{\widehat{x}}
\newcommand{\incore}{{\Omega_i}}
\newcommand{\outcore}{{\Omega_o}}
\newcommand{\core}{{\Omega_c}}
\newcommand{\exterior}{{\Omega_e}}
\usepackage{stmaryrd}
\DeclarePairedDelimiter{\jump}{\llbracket}{\rrbracket}

\begin{document}

\title{Existence of solutions for a model of the Earth's magnetic field}
\author[J. Bedrossian]{Jacob Bedrossian}
\address{Department of Mathematics, University of California, Los Angeles}
\email[]{jacob@math.ucla.edu}
\author[T. Schang]{Tom Schang}
\address{Department of Mathematics, University of California, Berkeley}
\email[]{tom\_schang@berkeley.edu}
\author[F. Weber]{Franziska Weber}
\address{Department of Mathematics, University of California, Berkeley}
\email[]{fweber@berkeley.edu}
\thanks{F.W. and T.S. were supported in part by NSF DMS-2438083, J.B. was supported in part by NSF DMS-2108633 and NSF DMS-2510949}


\begin{abstract}
    We study a  physically realistic, whole-core mathematical model of the dynamics in the Earth's core and we prove existence of Leray-Hopf type weak solutions to the model. Our model combines Magneto-Hydrodynamic equations in the liquid outer core with solid physics for the electrically conducting inner core, and treats everything exterior to the core as a perfect insulator governed by Maxwell's equations. We prove existence of weak solutions using Galerkin approximations. In order to control the nonlinearities, we must define an appropriate function space for the magnetic field and prove a Biot-Savart type result. The main new difficulty here is properly setting up the functional framework to simultaneously deal with the fluid structure interaction with the inner core and the magnetic transmission problem, with both the perfectly conducting inner core and the perfectly insulating mantle/exterior. 
\end{abstract}
\maketitle

\setcounter{tocdepth}{1}
\tableofcontents

\section{Introduction}

In 1919, Larmor proposed a model, which has developed into the geodynamo model, to explain how the Earth generates its magnetic field \cite{larmor_how_1919}. According to the geodynamo model, heat is slowly released from Earth's solid inner core to the liquid outer core. This heating creates a convection effect which, coupled with the Coriolis force from rotation of the Earth, produces helical currents in the liquid outer core. Because the outer core is composed of electrically conducting molten metals, the helical currents act as solenoids and create the magnetic field that we experience. While the dominant theory is that the convection due to buoyant forces related to temperature and composition of the fluid is the primary driving force, researchers are also exploring other mechanisms, such as tidal forces and precession \cite{landeau_sustaining_2022}.

More recently, researchers have simulated the geodynamo model numerically \cite{glatzmaier_simulating_1997,roberts_geodynamo_2000,Glatzmaiers1995}. Their experiments yielded results that support the geodynamo model -- the total magnetic energy is experimentally observed to sustain over time. Their experiments are also consistent with the observed phenomenon of \emph{geomagnetic reversals}, which is that the Earth's magnetic field flips direction at random times on the scale of hundreds of thousands of years~\cite{Merrill1983,Merrill1997,Clement2004}. Some recent simulations have focused on the effect of turbulence in the geodynamo~\cite{schaeffer_turbulent_2017}.

However, to our knowledge, there is no rigorous mathematical proof that the systems being simulated are mathematically well-defined in that they admit solutions. In particular, we do not know whether the numerical schemes from previous simulations converge to a solution of the system of equations they purport to solve. In this paper, we address this research gap by rigorously defining a system of equations that models the earth's magnetic field and the dynamics of the molten metals in the core, and proving existence of Leray-Hopf-type weak solutions to this system of equations for all times. We focus on weak solutions as we are interested in the long-time dynamics in non-perturbative regimes where it will not be possible to prove the existence of strong solutions beyond a short time.

\subsection{Model Description}
\begin{figure}
    \centering
    \begin{tikzpicture}[scale=1.3,>=stealth]
      \def\rinner{1.2}
      \def\router{2.4}

      \fill[pattern={Lines[angle=45,distance=8pt]}] (0,0) circle (\router);   
      \fill[white] (0,0) circle (\rinner);  
      \fill[pattern={Dots[distance=6pt]}] (0,0) circle (\rinner);   
    
      \draw[line width=1.2pt] (0,0) circle (\router); 
      \draw[line width=1.2pt] (0,0) circle (\rinner); 
    
      \node[fill=white] at (\rinner/2,0) {$\Omega_i$};                             
      \node[fill=white] at ({(\rinner+\router)/2},0) {$\Omega_o$};                 
      \node at ({\router + (\router-\rinner)/2},0) {$\Omega_e$};                    
    
      \draw[->,bend left=15,line width=1pt] 
        (-3,0.0) node[left] {$\partial \Omega_i$} to (180:\rinner);
      \draw[->,bend left=15,line width=1pt] 
        (-3,{\router*sin(160)}) node[left] {$\partial \Omega_c$} to (160:\router);
    
      \draw[-,thick] (0,0) -- ({\rinner*cos(115)},{\rinner*sin(115)}) node[midway,right,fill=white] {$R_i$};
      \draw[-,thick] (0,0) -- ({\rinner*cos(115)},{\rinner*sin(115)});
      \draw[-,thick] (0,0) -- ({\router*cos(125)},{\router*sin(125)}) node[pos=0.65,left,fill=white] {$R_o$};
      \draw[-,thick] (0,0) -- ({\router*cos(125)},{\router*sin(125)});
    
      \node[fill=white] at (0,-0.3*\rinner) {$\displaystyle \operatorname{div} B = 0$};
      \node[fill=white] at (0,-0.55*\rinner) {$\displaystyle u = \omega \times x$};
    
      \node[fill=white] at (0,-0.65*\router) {$\displaystyle \operatorname{div} B = 0$};
      \node[fill=white] at (0,-0.8*\router) {$\displaystyle \operatorname{div} u = 0$};
    
      \node at ({\router + (\router-\rinner)/2},-0.6*\router) {$\displaystyle \operatorname{div} B = 0$};
      \node at ({\router + (\router-\rinner)/2},-0.75*\router) {$\displaystyle \operatorname{curl} B = 0$};
      \node at ({\router + (\router-\rinner)/2},-0.9*\router) {$\displaystyle u = 0$};
    \end{tikzpicture}
    \caption{The domain over which the geodynamo problem is posed. The inner ball $\incore$ is the inner core, the annulus $\outcore$ is the outer core, and everything outside of that is the exterior $\exterior$.}
    \label{fig:setup}
\end{figure}
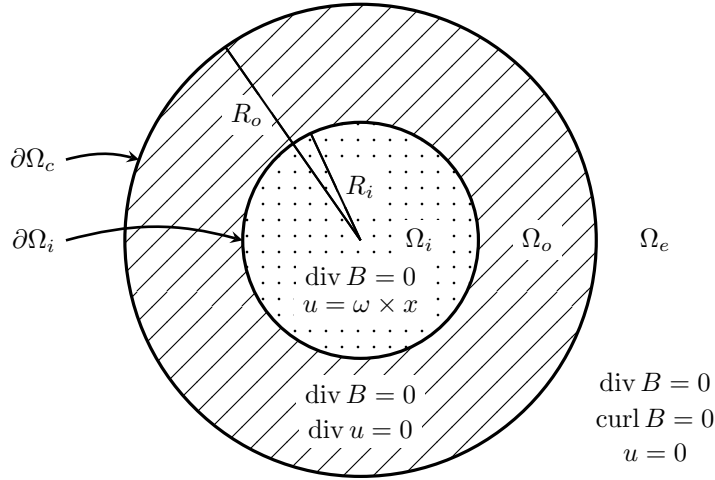

Our proposed model is largely derived from the model used in \cite{glatzmaier_simulating_1997}. To begin with, there are three domains of interest. The \emph{inner core} of the earth is given by $\incore := \{|x|< R_i\}$. The \emph{outer core} is given by $\outcore:=\{ R_i<|x|<R_o\}$. The \emph{core} is $\core :=\{|x|<R_o\}$. 
We also define everything \emph{exterior} to the core as $\exterior = \{|x|>R_o\}$. A graphical depiction can be found in \Cref{fig:setup}.

Physically, the liquid outer core is an electrically conducting liquid with a buoyant force proportional to the temperature and composition of the fluid. We model it using the \gls{mhd} equations in a rotating frame. In particular, we intend to solve for the fluid velocity $u$, the magnetic field $B$, the buoyancy variable (combining effects of composition and temperature on buoyancy) $\theta$, and the fluid pressure $p$. In the outer core $\outcore$, they must satisfy the following equations
\begin{subequations}\label{eqs:outer-core}
    \begin{equation}
        \label{eq:momentum-outer-core}
        \partial_t u +(u\cdot \nabla) u +\nabla p = \nu\Delta u + \rho^{-1}\mu^{-1}\tcurl B\times B + \theta g \radial - Le_3\times u
    \end{equation}
    \begin{equation}
        \label{eq:magnetic-outer-core}
        \partial_t B -\tcurl(u\times B) =-\tcurl\p*{\sigma^{-1} \tcurl \mu^{-1} B}
    \end{equation}
    \begin{equation}
        \label{eq:temp-outer-core}
        \partial_t\theta + (u\cdot \nabla)\theta =\kappa\Delta \theta + f
    \end{equation}
    \begin{equation}
        \label{eq:incompressible}
        \tdiv u = 0 
    \end{equation}
    \begin{equation}
        \label{eq:no-monopoles-outer-core}
        \tdiv B = 0 
    \end{equation}
\end{subequations}
where $\nu$ is the fluid viscosity, $\rho$ is the fluid density, $g$ is the acceleration due to gravity, $\radial$ is the outwards pointing radial unit vector (i.e., $\radial = x/|x|$), $Le_3$ is the angular velocity of the rotating frame (which represents the angular velocity of the Earth), $\mu$ is the magnetic permeability, $\sigma$ is the electric conductivity, and $f$ is taken to be a prescribed external source term on the buoyancy variable, capturing effects such as radioactive decay. In a future work, we will consider a stochastic forcing $f$ to model small-scale turbulence. \Cref{eq:momentum-outer-core} is the conservation of momentum equation for the fluid velocity we expect from Navier-Stokes, including additional terms for the Lorentz force, the buoyant force, and the Coriolis force. \Cref{eq:magnetic-outer-core} is the induction equation from the \gls{mhd} equations. \Cref{eq:temp-outer-core} is an advection-diffusion equation describing the evolution of the buoyancy variable.

The inner core is modeled as a rigid electrically conducting ball, so that on $\incore$ we solve for the relative angular velocity $\omega$ (which determines the velocity field $u$), and the magnetic field $B$ that satisfy the following equations
\begin{subequations}\label{eqs:inner-core}
    \begin{equation}\label{eq:angular-inner-core}
        J\frac{\d }{\d t}\omega = \int_{\partial \incore}  R_i n\times (\rho\Sigma\cdot n) \d S-JLe_3\times \omega
    \end{equation}
    \begin{equation}\label{eq:momentum-inner-core}
        u:=\omega\times x
    \end{equation}
    \begin{equation}\label{eq:magnetic-inner-core}
        \partial_t B -\tcurl \p*{u \times B} = -\tcurl (\sigma^{-1}\tcurl\mu^{-1} B)
    \end{equation}
    \begin{equation}\label{eq:divergence-inner-core}
        \tdiv B = 0
    \end{equation}
\end{subequations}
where $J$ is the angular momentum of the inner core, and 
\[\Sigma = -pI+2\nu D^S(u)+\mu_o^{-1}\rho^{-1}|B|^2 I + \mu_o^{-1}\rho^{-1}B\otimes B\]
is the fluid stress tensor. We define $D^S(u) = \frac{1}{2}(\nabla u +\nabla u^T)$ to be the symmetrized gradient.
\Cref{eq:angular-inner-core} is the conservation of angular momentum, where the surface integral is the torque applied by the liquid outer core on the surface of the inner core, and the second term is the Coriolis effect. \Cref{eq:magnetic-inner-core} is again the induction equation from the \gls{mhd} system.

Finally, we must extend the magnetic field outside of the core. We treat the region exterior to the core as a perfect insulator (i.e., $\sigma=0$), so that on $\exterior$ we solve for the magnetic field $B$ and electric field $E$ satisfying 
\begin{subequations}\label{eqs:exterior}
    \begin{equation}
        \partial_t B = -\tcurl E
    \end{equation}
    \begin{equation}
        \tdiv B = 0
    \end{equation}
    \begin{equation}
        \tcurl B =  0
    \end{equation}
    \begin{equation}
        \tdiv E = 0.
    \end{equation}
\end{subequations}
This is in contrast to the inner and outer core where the electric field is eliminated using Ohm's law. The need to explicitly include the electric field can also be motivated by formally considering \Cref{eq:magnetic-outer-core} with $u=0$, $\sigma=0$, and $\tcurl (\mu^{-1} B)=0$.
The electric field is the correct way to interpret the indeterminate term that appears inside the curl. This effectively produces a non-local boundary condition for the magnetic field on the core-mantle boundary, and is the source of most of the new complexity in this work.

To make this model realistic, we allow for piecewise constant physical constants, 
\[
\mu = \begin{cases}
    \mu_i   & \text{in }\incore \\
    \mu_o & \text{in }\outcore \\
    \mu_e & \text{in }\outcore
\end{cases}, \quad 
\sigma = \begin{cases}
    \sigma_i &\text{in }\incore \\
    \sigma_o&\text{in }\outcore \\
    0 & \text{in }\exterior
\end{cases}.
\]
The model also must include how the quantities at each interface are coupled. At the interface between the inner and outer core $\partial\incore$,
\begin{subequations}\label{eqs:BC-inner}
    \begin{equation}\label{eq:inner-BC-no-slip}
        u = \omega\times R_i n
    \end{equation}
    \begin{equation}\label{eq:inner-BC-heat-flux}
        \nabla \theta(t, x)\cdot n = f_b(x)
    \end{equation}
    \begin{equation}\label{eq:inner-BC-magnetic-normal}
        \llbracket B\cdot n\rrbracket =0
    \end{equation}
    \begin{equation}\label{eq:inner-BC-magnetic-tangential}
        \jump{\proj_{n\perp}\mu^{-1}B}=0
    \end{equation}
    \begin{equation}\label{eq:inner-BC-neumann}
        \llbracket\sigma^{-1}\tcurl (\mu^{-1}B)\times n\rrbracket = 0
    \end{equation}
\end{subequations}
where $n$ is the normal vector pointing outwards from $\incore$, $\proj_{n\perp}$ is tangential projection (i.e., $\proj_{n\perp } y = y- (y\cdot n)n$). We use $\jump{\cdot }_{S}$ to indicate the jump of a quantity across a surface $S$ in the direction of the implied normal $n$. We omit the surface when it is clear from context.
\Cref{eq:inner-BC-no-slip} is a no-slip boundary condition between the fluid velocity and the boundary of the solid inner core. \Cref{eq:inner-BC-heat-flux} is a is a flux boundary condition with prescribed right-hand-side $f_b$, modeling the release of heat and lighter elements from the inner core to the outer core. \Cref{eq:inner-BC-magnetic-normal,eq:inner-BC-magnetic-tangential} are transmission conditions coming from refraction in Maxwell's equations, in particular they are the conditions $\jump{B\cdot n}=\jump{H\times n}=0$ (where $B = \mu H$). The additional boundary condition \cref{eq:inner-BC-neumann} represents the jump condition $\jump{E\times n}=0$. 

At the interface between the outer core and the exterior $\partial \core$ (physically, the interface with the mantle), we impose
\begin{subequations}\label{eqs:BC-outer}
    \begin{equation}
        u = 0 
    \end{equation}
    \begin{equation}\label{eq:outer-BC-temperature}
        \theta(t,x) = \theta_b(x) 
    \end{equation}
    \begin{equation}
        \jump{B\cdot n} = 0
    \end{equation}
    \begin{equation}\label{eq:outer-BC-magnetic-tangential}
        \jump{\proj_{n\perp}\mu^{-1}B} = 0
    \end{equation}
    \begin{equation}\label{eq:outer-BC-neumann}
        E\times n = \sigma^{-1}\tcurl \mu^{-1}B\times n
    \end{equation}
\end{subequations}
where $n$ is the normal vector pointing outwards from $\core$. Similarly to the equations on $\partial\incore$, we have a no-slip condition on $u$, a Dirichlet condition for the bouyancy variable given by $\theta_b$, as well as refraction transmission conditions on the magnetic and electric fields.

\subsection{Main Result}
Our main result is the following:
\begin{theorem}\label{thm:existence-intro}
    Given initial data $B_0,(u_0,\omega_0),\theta_0$ in an appropriate space, and given $f,f_b,\theta_b$ also in an appropriate space, then for any time $T$, there exists a weak solution $(B,(u,\omega),\theta)$ defined on $[0,T]$ to \cref{eqs:outer-core,eqs:inner-core,eqs:exterior,eqs:BC-inner,eqs:BC-outer}.
\end{theorem}

\subsection{Related Work}
Various aspects of the problem at hand have already been studied. In the realm of fluid-structure interaction, weak solutions have been shown to exist for a rigid body immersed in a viscous incompressible fluid, up to the time of first contact between the solid and the fluid boundary \cite{conca_existence_2000}. We use a similar notion of weak solutions as in that work. The motion of rigid bodies submerged in incompressible fluids has also been studied by: Ortega, Rosier, and Takahashi~\cite{Ortega2005}, where global in time existence of classical solutions was shown for a ball emerged in a perfect incompressible fluid in 2D; by Ducomet and Ne\v{c}asov\'{a}~\cite{Ducomet2013}, where the motion of rigid bodies in  incompressible and compressible viscous
fluids under the action of gravitational forces was studied; and by B\u{a}lilescu, San Mart\'{i}n and Takahashi~\cite{Balilescu2017}, where the existence of weak solutions for a model of ``a rigid body and a viscous incompressible fluid using a boundary condition based on Coulomb's law'' was shown. The resistive \gls{mhd} equations have also been well-studied, such as  in the works of: Duvaut and Lions~\cite{Duvaut1972} where the existence and uniqueness of weak and strong solutions was studied; Sermange and Temam~\cite{Sermange1983} where the long time behavior of solutions was studied; and Cheskidov and Dai~\cite{Cheskidov2025} where regularity criteria for solutions are studied, to mention a few.

A similar system to ours is analyzed in \cite{benesova_fluid-rigid_2023}. That work considers a rigid insulating body immersed in a viscous, incompressible, electrically conducting fluid, with perfect conductor boundary conditions on the boundary between the fluid and the exterior. For this system, the authors prove existence of weak solutions up to the time of first contact between the solid and the fluid boundary. This result has also been extended to the case of multiple rigid bodies immersed in a compressible fluid \cite{scherz_fluidrigid_2023,scherz_schwache_2024}. A rigorous derivation of these systems and boundary conditions can be found in \cite{scherz_modeling_2024}. In contrast, we consider a setting in which the rigid body is conductive and the region external to the fluid is a perfect insulator which leads to a different problem for the unbounded region. We also consider a rotating but not translating rigid body since the translation is not relevant for the dynamics in the core of the earth. We do, however, allow for discontinuous, piecewise constant physical parameters instead of strictly constant ones as in~\cite{benesova_fluid-rigid_2023} due to the inner core, outer core, and the mantle having different physical properties. In addition, the proofs in these previous works appear more complicated and involve both Rothe's method and Brinkman penalization, whereas we achieve our results using only a spatial discretization.

The problem studied here falls into the wider area of electromagnetic transmission problems. A variety of magnetic transmission problems relevant for different physical scenarios have been studied in the literature. A review can be found in~\cite{meir_variational_1996}, where the authors distinguish three basic types found in~\cite{ladyzhenskaya_solution_1960,Ladyshenskaya1959}: (a) the \emph{perfect conductor} problem, where a fluid is confined to a finite vessel with perfectly conducting walls, resulting in an initial-boundary value problem, and (b) the \emph{matching} problem, where a bounded region of electrically conducting fluid is coupled to another region via jump conditions. A typical problem  of this type is a situation with a ``bounded region of conducting fluid surrounded by an insulator of infinite extent.'' There are also (c) \emph{mixed} problems, where two regions are coupled via jump conditions and are contained themselves in a larger region with perfectly conducting walls. For example, a ``conducting fluid and a current carrying solid conductor embedded in a region of insulating material bounded by perfectly conducting walls'' would fall into this category. The problem we are considering in this work has aspects of types (b) and (c), as it concerns two conducting regions coupled via jump conditions and surrounded by an infinite insulator. Problems of type (b) and (c) were first studied in~\cite{ladyzhenskaya_solution_1960,Ladyshenskaya1973,Solonnikov1960}.
More recent progress has been made both on the theoretical (e.g.,~\cite{melenk_regularity_2025}), as well as the numerical side (e.g.,~\cite{meir_analysis_1999,iskakov_magnetic_2005}). In~\cite{melenk_regularity_2025}, the authors investigate higher regularity for vector fields with piecewise regular curl and/or divergence, which are precisely the sort of vector fields encountered in matching problems. In~\cite{iskakov_magnetic_2005}, the authors define a boundary element method for treating matching problems numerically without having to extend the magnetic field to an insulator of infinite extent. 

In our work here, we combine for the first time aspects of magnetic transmission problems of type (b) and (c) such as the matching boundary conditions and the infinite insulating exterior, with rigid body dynamics and the resistive \gls{mhd} equations. To do so, we first define a suitable notion of weak solutions inspired by~\cite{conca_existence_2000} that avoids the boundary integral of the stress tensor in~\eqref{eq:angular-inner-core} which may otherwise not be regular enough for weak solutions. Then we use a Galerkin approach to set up finite dimensional approximate systems for~\eqref{eqs:outer-core}-\eqref{eqs:BC-outer} that satisfy an energy balance yielding a priori estimates. The main challenge here is to define a suitable Hilbert space for the approximation of the magnetic field $B$ leading to the correct transmission conditions. This requires an appropriate Biot-Savart law with refraction. Then we pass to the limit in the approximating system using the Aubin-Lions lemma and conclude that the limit is a weak solution of~\eqref{eqs:outer-core}-\eqref{eqs:BC-outer}.

\subsection{Outline}
The rest of this article is structured as follows: In Section~\ref{sec:functionspace} we construct the necessary function spaces, in particular for the magnetic field variable. We prove a Biot-Savart law with refraction that allows us to define a suitable space for the magnetic field. In Section~\ref{sec:weaksol}, we define weak solutions for~\eqref{eqs:outer-core}-\eqref{eqs:BC-outer} and prove their existence using a suitable Galerkin approximation. Finally, the appendix~\ref{sec:appendix} is dedicated to technical auxiliary results.

\subsection{Notation}
In this paper we denote by $C$ or $C_i$ an arbitrary (large) constant, whose value may be different each time it appears. This occurs when we need to accurately keep track of one of the terms in an inequality but not the others. When it does not matter, we use the notation $\lesssim$. We use $\sim$ when both $\lesssim$ and $\gtrsim$ are true.
We will also use standard function spaces. The function space $\cc C^n(\Omega)$ consists of $n$-times continuously differentiable functions on the set $\Omega$, while $\cc C_0^n(\Omega)$ are in addition compactly supported in $\Omega$. We also use the Sobolev spaces $H^\alpha(\Omega)$, $H_0^\alpha(\Omega)$, and the divergence-free Sobolev space $H^1_{\tdiv}(\Omega)$. 
Product spaces, such as $X_1(\Omega_1)\times X_2(\Omega_2)\times\cdots X_n(\Omega_n)$ (where $\Omega_j$ are mutually disjoint) are endowed with the norm 
\[
\norm{u}_{X_1(\Omega_1)\times X_2(\Omega_2)\times\cdots X_n(\Omega_n)}^2 = \norm{u}_{X_1(\Omega_1)}^2 +\norm{u}_{X_2(\Omega_2)}^2+\cdots+\norm{u}_{X_n(\Omega_n)}^2.
\]
The inner product of two elements in an inner product space $X$ is denoted by $\langle\cdot,\cdot\rangle_X$.
We also overload notation so that $\langle \cdot, \cdot\rangle_\Omega$ is the $L^2(\Omega)$ inner product when $\Omega$ is a domain rather than a function space. Similarly, the duality pairing $\langle \cdot, \cdot\rangle_{X',X}$ is the duality pairing between a function space $X$ and its dual $X'$. Finally, we will sometimes use the shorthand $L^p_tX$ to mean the space of $X$-valued functions $L^p([0,T];X)$ (where $T$ is an unspecified arbitrary time).

\subsection{Acknowledgments}
F.W. and T.S. were supported in part by NSF grant DMS-2438083, and J.B. was supported by NSF grants DMS-2108633 and DMS-2510949. ChatGPT was used on September 25, 2025 to create tikz code for \Cref{fig:setup}, which was later refined by the authors.
\section{Function spaces}
\label{sec:functionspace}
In this section, we introduce the function spaces in which we will search for the solution to the geodynamo problem. We take the time to prove embeddings that will be important later.

For the bouyancy variable, we must address the non-homogeneous boundary conditions. We will control the Dirichlet condition \cref{eq:outer-BC-temperature} by 
subtracting an extension of the boundary condition, i.e., we look for $\tau$ where $\theta = \tau+\theta^c$ and $\theta^c(x)$ is an extension of the Dirichlet condition $\theta_b$ on $\partial \core$ and which satisfies the Neumann condition $\nabla\theta^c\cdot n = 0$ on $\partial \incore$. We therefore define
\[
\mathbb{T}^1 = \left\{\tau\in H^1(\outcore ): \tau\big|_{\partial \core} = 0\right\}
\quad \text{and}\quad 
\bb T^0 = L^2(\outcore)
\]
and they are given the standard $H^1$ and $L^2$ norms.
We also introduce the bilinear form 
\[a(\tau,\varphi) = \langle \nabla \tau,\nabla \varphi\rangle_{L^2(\outcore)}\]
and we observe that the standard Poincar\'{e} inequality 
\[ 
\norm{\tau}_{L^2(\outcore)}\lesssim\norm{\nabla\tau}_{L^2(\outcore)} 
\]
is true in $\bb T^1$ due to the homogeneous Dirichlet condition on $\partial\outcore$.

We combine the fluid velocity $u$ and angular velocity $\omega$ into one vector space (a common technique in fluid-structure interaction problems, see for example \cite{conca_existence_2000}), 
\[\mathbb{U}^1=\left\{(u,\omega)\in H^1_{\tdiv}(\outcore)\times \R^3: u\big|_{\partial \core} = 0, u\big|_{\partial \incore} = \omega\times x\right\}\]
and
\[\bb U ^0 = \overline{\bb U ^1}^{L^2\times \R^3}.\]
Combining both variables is important since it will enable a good energy estimate for the approximating solutions.
We define the inner product in $\bb U^0$ to be 
\[
\langle (u,\omega),(v,\alpha)\rangle_{\bb U^0} = \rho \langle u,  v\rangle_{L^2(\outcore)} + J\omega\cdot\alpha
\]
and in $\bb{U}^1$ the inner product is defined by 
\[
\langle (u,\omega),(v,\alpha)\rangle_{\bb U^1} = \langle (u,\omega),(v,\alpha)\rangle_{\bb U^0} + 2\rho\nu a^S(u,v)
\]
where 
\[
a^S(u,v) = \langle D^S(u),D^S(v)\rangle_{L^2(\outcore)}.
\]
\begin{lemma}\label{lem:poincare-korn}
    Let $(u,\omega) \in \mathbb{U}^1$. Then $\norm{ u }_{H^1(\outcore)}\sim \norm{ u }_{L^2(\outcore)}+\norm{D^S(u)}_{L^2(\outcore)}$. In particular
    \[
    \norm{u}_{H^1(\outcore)}\lesssim 
    \norm{u}_{L^2(\outcore)} + \norm{D^S(u)}_{L^2(\outcore)}.
    \]
    More strongly, $\norm{D^S(u)}_{L^2(\outcore)} \sim \norm{\nabla u}_{L^2(\outcore)}$, which implies 
    \[ \norm{u}_{H^1(\outcore)}\lesssim\norm{D^S(u)}_{L^2(\outcore)}.
    \]
    \begin{proof}
        The weaker statement is Korn's inequality (e.g., Theorem 2.1 in~\cite{ciarlet_korns_2010}).

        For the stronger statement, consider $u$ on the entire core $\core$, where we extend to $\incore$ according to \cref{eq:momentum-inner-core}. Since this extended $u\in H_0^1(\core)$, we can apply the homogeneous version of Korn's inequality $\norm{\nabla u}_{L^2(\core)}\lesssim \norm{D^S(u)}_{L^2(\core)}$ (e.g., Chapter 3 in \cite{acosta_korns_2017}) and the standard Poincar\'{e} inequality as follows,
        \[\norm{u}_{H^1(\outcore)}\leq 
        \norm{u}_{H^1(\core)}\lesssim \norm{\nabla u}_{L^2(\core)}\lesssim\norm{D^S(u)}_{L^2(\core)}=\norm{D^S(u)}_{L^2(\outcore)}.\]
        The last equality comes from the fact that $\nabla(\omega\times x)$ is antisymmetric.
    \end{proof}
\end{lemma}

For the magnetic field, consider the spaces 
\[\bb B^1 = \set*{ B\in H_{\tdiv}^1(\incore)\times H_{\tdiv}^1(\outcore)\times H_{\tdiv}^1(\exterior): \tdiv B \in L^2(\R^3), \tcurl (\mu^{-1} B)\in L^2(\R^3), \tcurl B |_{\exterior}=0 }\]
and 
\[\bb B^0 = \overline{\bb B ^1 }^{L^2}.\]
We denote by $B_i$ the restriction of $B$ on $\incore$ (and analogous notation for the regions $\outcore,\exterior$). By $B$ we may mean either the triple $(B_i,B_o,B_e)$ or the function $\bbm{1}_\incore B_i+\bbm{1}_\outcore B_o+\bbm{1}_\exterior B_e$ depending on the context.

Observe that (as distributions)
\[\tcurl (\mu^{-1} B) \in L^2(\bb R^3)\iff \jump{\mu^{-1}B\times n}_{\partial \incore} = \jump{\mu^{-1}B\times n}_{\partial \core} = 0\]
and
\[\tdiv B \in L^2(\bb R^3)\iff \jump{B\cdot n}_{\partial \incore} = \jump{B\cdot n}_{\partial \core} =0 \]
(see, for example, \cite{meir_variational_1996}) so that our definition of $\bb B^1$ incorporates the zero-th order boundary conditions on the magnetic field. 

We define the inner product in $\bb B^0$ to be 
\[
\langle B, A\rangle _{\bb B^0} = \langle B, \mu^{-1} A\rangle_{L^2(\R^3)}
\]
and in $\bb B^1$ to be 
\[
\langle B, A\rangle _{\bb B^1} = \langle B, A\rangle_{\bb B^0} + d(B,A),
\]
where 
\[d(B,A) = \langle \sigma^{-1} \tcurl \mu^{-1} B, \tcurl \mu^{-1} A\rangle_{L^2(\core)}.\]
As before, we state the relevant function space embedding.
\begin{lemma}\label{lem:curl-norm}
    Let $B\in \bb{B}^1$. Then $\norm{B}_{ H^1(\incore)\times H^1(\outcore)\times H^1(\exterior)}\sim \norm{B}_{\bb B^1}$. In particular,
    \[
    \norm{B}_{H^1(\incore)\times H^1(\outcore)} \lesssim \norm{B}_{\bb B^1}.
    \]
\end{lemma}
Because it requires a lot of exposition, the proof is postponed to the next subsection (\Cref{subsec:embeddings}). Theorem 9 in~\cite{ladyzhenskaya_solution_1960}, and Corollary 2.10 in~\cite{meir_variational_1996} provide similar results, but both are restricted to cases where there is only one interface.

\begin{lemma}\label{prop:magnetic-hilbert}
    The space $\mathbb{B}^1$ is a (separable) Hilbert space. 
\end{lemma}
\begin{proof}
    Define the continuous linear operator $T:H^1_{\tdiv}(\incore)\times H^1_{\tdiv}(\outcore)\times H^1_{\tdiv}(\exterior)\to H^{1/2}(\partial\incore)\times H^{1/2} (\partial \core) \times L^2(\exterior)$ by 
    \[
    T:B\mapsto \p*{\jump*{(B\cdot n)n+\proj_{n\perp}(\mu^{-1} B) }_{\partial\incore} ,\jump*{(B\cdot n)n+\proj_{n\perp}(\mu^{-1} B) }_{\partial\outcore}, \left.\tcurl B\right|_{\exterior}}.
    \]
    It is clear that $\bb B^1 = \ker T$ as sets. The kernel of a bounded linear operator is a subspace. Separability follows from the fact that we have a subspace of a separable vector space where the norm is equivalent or unchanged. 
\end{proof}

\subsection{Magnetic field embedding}\label{subsec:embeddings}
In this section, we prove \Cref{lem:curl-norm}. First, we need to introduce relevant background. We require three key ingredients: the Beppo-Levi space, piecewise regularity of the refraction problem, and the Biot-Savart law.

In what follows, let $\Omega_-$ be a connected compact domain in $\mathbb R^3$ with smooth boundary and connected exterior $\Omega_+$, and $\mu$ be piecewise constant on $\Omega_-$ and $\Omega_+$. 

\subsubsection{Beppo-Levi space}
Define the space $\dot H^1(\Omega_+)$ to be the completion of compactly supported smooth functions in $\overline{\Omega}_+$ under the homogeneous $\dot H^1$ norm, i.e. under the norm
\[\norm{\phi}_{\dot{H}^1(\Omega_+)}=\norm{\nabla \phi}_{L^2(\Omega_+)}.\]
It is known that in $\dot H^1(\Omega_+)$ this norm is equivalent to the norm 
\[
\p*{\norm{\nabla \phi }_{L^2(\Omega_+)}^2+\norm{(1+|x|^2)^{-1/2}\phi}_{L^2(\Omega_+)}^2}^{1/2}
\]
(e.g., Proposition 2.10.10 in~\cite{sauter_boundary_2011}) and this space is called the \emph{Beppo-Levi} space. For functions in $\dot H^1(\Omega_+)$, observe that on compact subsets $\Omega'\subset \Omega_+$, the $L^2(\Omega')$ norm is controlled by the $\dot H^1(\Omega_+)$ norm.

\subsubsection{Refraction problem} The refraction problem arises when we have elliptic-type equations in divergence form with a piecewise constant coefficient. What is important to us is that solutions are piecewise-regular up to the boundary.
\begin{proposition}\label{prop:refraction}
    Let $\Omega_+$ and $\Omega_-$ be non-empty connected compact domains in $\R^3$ such that $\Gamma=\overline{\Omega}_+\cap \overline {\Omega}_-$ is a smooth codimension-1 surface, and $S = \partial (\overline{\Omega}_+\cup \overline{\Omega}_-)$ is also a smooth codimension-1 surface, and such that $\Gamma\cap S$ is empty. Let $\mu=\mu_\pm$ on $\Omega_\pm$ where $\mu_+,\mu_-$ are positive constants. Consider the refraction problem
    \begin{equation}\label{eq:refraction}
        \begin{cases}
            \tdiv (\mu \nabla \phi ) = f & \Omega_\pm
            \\
            \jump{\phi } = 0 & \Gamma
            \\
            \jump{\mu\nabla \phi\cdot n} = 0 & \Gamma
            \\
            \phi = 0 & S
        \end{cases}.
    \end{equation}
    Suppose $f\in H^s(\Omega_+)\times H^s(\Omega_-)$ for $s>-1$. Then there is a unique weak solution $\phi\in H^1(\Omega_+\cup\Gamma\cup \Omega_-)$ to \cref{eq:refraction}. Moreover, $\phi\in H^{s+2}(\Omega_+)\times H^{s+2}(\Omega_-)$ and it satisfies the estimate
    \[ \norm{\phi}_{H^{s+2}(\Omega_+)\times H^{s+2}(\Omega_-)}\lesssim \norm{f}_{H^s(\Omega_+)\times H^s(\Omega_-)}. \]
\end{proposition}
This is exactly Theorem 1.1 in~\cite{babuska_finite_1970}. A similar result for the specific case where $f\in L^2$ can be found with proof as Lemma 1 in~\cite{ladyzhenskaya_solution_1960}.

\subsubsection{Biot-Savart law}
We define the Newton potential in $\R^3$
    \[G(x) =\frac{-1}{4\pi |x|}\]
and the operator of convolution with the Newton potential $\cc L_0: L^2(\bb R^3)\to H^2(\bb R^3)$ by
    \[\cc L_0 \phi (x) := \int_{\R^3} G(x-y)\phi(y)\d y.\]
\begin{lemma}\label{lem:div-curl}
    Given $h\in L^2(\R^3)$, the vector field $f=-\tcurl \cc L_0h$ is the unique solution in $\dot H^1(\R^3)$ of 
    \[
    \begin{cases}
        \tdiv f = 0 & \R^3\\
        \tcurl f = h & \R^3
    \end{cases}
    \]
    and is bounded as a map $L^2_{\tdiv}(\R^3)\to \dot H^1(\R^3)$. Moreover, if $h$ has compact support, then $f\in H^1(\R^3)$.
\end{lemma}
This is Lemma 2.2 (b) in~\cite{meir_variational_1996}.

\subsubsection{Biot-Savart with refraction.}

\begin{proposition}\label{lem:biot-savart-with-refraction}
    Let $B\in \bb B^1$. Then
    \begin{equation}\label{eq:biot-savart-with-refraction}
        \norm{B}_{\dot{H}^1(\incore)\times \dot{H}^1(\outcore)\times \dot{H}^1(\exterior)} \lesssim \norm{\tcurl \mu^{-1} B}_{L^2(\core)}.
    \end{equation}
\end{proposition}
\begin{proof}
    We will follow the technique of Theorem 2.7 in~\cite{meir_variational_1996}. Our strategy will be to show that for any $h\in L^2_{\tdiv}(\core)$ there exists a unique $f\in\bb B^1$ such that $\tcurl (\mu^{-1}f) = h$, and that this $f$ satisfies the desired estimate. To put it in different words, we pretend that we only know $\tcurl (\mu^{-1} B) = h$, and we construct a function $f$ such that $\tcurl (\mu^{-1} f )= h$ and such that the desired inequality is satisfied. Finally, we prove that this $f$ is unique, so that it is in fact $B$ and therefore $B$ satisfies the desired inequality.
    
    Extending $h$ to be zero outside $\core$, we define $f_1:= -\tcurl \cc L_0 h$. By \Cref{lem:div-curl}, we know $f_1\in \dot H^1(\R^3)$ (in fact, $H^1(\R^3)$ because $h$ has compact 
    support) and
    \[\tdiv f_1 = 0,\qquad \tcurl f_1 = h\]
    as functions in $L^2(\R^3)$. Moreover, 
    \[\norm{f_1}_{\dot H^1(\R^3)}\lesssim \norm{h}_{L^2(\R^3)}\]
    and thus, taking the trace,
    $f_1\cdot n\in H^{1/2}(\Gamma)$ where $\Gamma$ is any smooth codimension 1 surface. 
    If we now set $f = \mu f_1$, then we find that 
    \[\tcurl (\mu^{-1} f) = \tcurl f_1 = h\]
    which fulfills some of our requirements, but on the other hand 
    \[\tdiv f \not\in L^2(\R^3)\] 
    because 
    \[\jump{f\cdot n}_{\partial\incore,\partial\core} = \jump{\mu f_1\cdot n}\]
    which is in general nonzero, because $\jump{f_1\cdot n}$ is zero and the trace of $f_1\cdot n$ is in general nonzero. This means that we need to correct this error in the jump condition. We look for a decomposition $f=\mu(f_1+f_2)$.
    In particular, we need to find $f_2$ solving 
    \[
    \begin{cases}
        \tcurl f_2 = 0 & \R^3 
        \\
        \tdiv f_2 = 0 &  \incore,\outcore,\exterior 
        \\
        \jump{\mu f_2\cdot n} = -\jump{\mu f_1\cdot n} &  \partial\incore,\partial\core\,.
    \end{cases}
    \] 
    Instead of solving this directly, we look for a solution $f_2 = \nabla \phi$ where $\phi\in \dot H^1(\R^3)$ satisfies $\nabla ^2\phi\in L^2(\bb R^3)$ and solves 
    \[
    \begin{cases}
        \tdiv (\mu \nabla \phi) = 0 & \incore,\outcore,\exterior
        \\
        \jump{\mu \nabla \phi\cdot n} = -\jump{\mu f_1\cdot n} &   \partial \incore,\partial\core\,.
    \end{cases}
    \]
    Observe that the right hand side of the jump condition $\jump{\mu f_1\cdot n}$ is in the space $H^{1/2}(\partial\incore)$ and $H^{1/2}(\partial\outcore)$. This means that we can find a bounded, compactly supported, piecewise $H^2$ extension of this jump condition. 
    In particular, we can choose $g_o\in H^2(\incore)\times H^2(\outcore)\times H^2(\exterior)$ with $\supp g_o \subseteq R_o-\delta \leq |x|\leq R_o$ (with $\delta<R_o-R_i$) such that $\jump{\mu\nabla g_o\cdot n}_{\partial \core}= -\jump {\mu f_1\cdot n}_{\partial\core}$. 
    Similarly, we can choose $g_i\in H^2(\incore)\times H^2(\outcore)\times H^2(\exterior)$ with $\supp g_i\subseteq \overline \incore$ such that $\jump{\mu\nabla g_i\cdot n}_{\partial \incore}= -\jump {\mu f_1\cdot n}_{\partial\incore}$. We can now define $\phi = \widetilde \phi + g_i + g_o$ and define a problem for $\widetilde{\phi}$:
    \[
    \begin{cases}
        \tdiv(\mu\nabla \widetilde \phi) =-\tdiv (\mu \nabla g_i)-\tdiv (\mu \nabla g_o) & \incore,\outcore,\exterior \\
        \jump{\mu\nabla \widetilde \phi\cdot n} = 0 & \partial\incore,\partial\outcore\,.
    \end{cases}
    \]
    Applying Lax-Milgram to the variational formulation in the space $\dot H^1(\R^3)$ 
    \[\int_{\R^3}\nabla \widetilde \phi \cdot \nabla \psi\d x = \int_{\R^3}\p*{\tdiv (\mu \nabla g_i) + \tdiv (\mu \nabla g_o)} \psi\d x\]
    yields a solution $\widetilde \phi \in \dot H^1(\R^3)$ satisfying 
    \[\norm{\nabla\widetilde  \phi }_{L^2(\R^3)}\lesssim \norm {g_i}_{H^2(\incore)}+\norm{g_o}_{H^2(\outcore)}.\]
    However, we also require $\widetilde \phi \in \dot H^2(\incore)\times\dot H^2(\outcore)\times \dot H^2(\exterior)$. In order to show that our solution has this higher piecewise regularity, we will localize to three different regions.
    \begin{enumerate}
        \item We localize near infinity. Take $\chi \in \cc C^\infty(\R)$ such that $\chi =1$ for $r\geq R_o+2$, $\chi=0$ for $r\leq R_o+1$, and with bounded derivative (say $\abs*{\chi '}\leq C$). Since on this region, $\mu = \mu_e$ is constant, by considering the Poisson equation solved by $\chi(|x|)\widetilde \phi$
        \[
        \Delta(\chi(|x|)\widetilde \phi) = \widetilde \phi\Delta\chi(|x|) +2\nabla \chi(|x|)\cdot\nabla \widetilde \phi+\chi(|x|)\Delta \widetilde \phi
        \]
        we get that 
        \[\norm{\widetilde \phi}_{\dot H^2(\{|x|\geq R_o+2\})}\lesssim \norm{\nabla \widetilde \phi}_{L^2(\R^3)}+\norm{\widetilde \phi}_{L^2(\{|x|\in\supp \Delta \chi\})}\lesssim \norm{\widetilde \phi}_{\dot H^1(\R^3)}.\]
        \item We localize near $\partial \core$. This time take $\chi=1$ on $[(R_o+R_i)/2, R_o+2]$ and $\chi=0$ for $r\geq R_o+3$ and $r\leq(R_o+3R_i)/4$. We can then apply the existence and regularity result \Cref{prop:refraction} for the refraction problem
        \[
        \begin{cases}
            \tdiv (\mu \nabla (\chi(|x|)\widetilde \phi)\,) = \mu \widetilde \phi\Delta \chi + 2\mu \nabla \widetilde \phi \cdot \nabla \chi +\chi \tdiv(\mu \nabla \widetilde \phi) & \exterior\cap \{|x|\leq R_0+3\},\core
            \\
            \jump{\chi(|x|)\widetilde \phi} = 0  &\partial\core
            \\
            \jump{\mu\nabla (\chi(|x|)\widetilde \phi )\cdot n  } = 0 & \partial \core
            \\
            \chi(|x|)\widetilde \phi = 0 & |x| = R_0+3
        \end{cases}
        \]
        to conclude
        \begin{multline}
            \p*{\norm{\widetilde \phi}^2_{\dot H^2(\exterior\cap \{|x|\leq R_o+2\})}+\norm{\widetilde \phi}^2_{\dot H^2(\outcore\cap \{ |x|\geq (R_o+R_i)/2\})}}^{1/2}
            \\
            \lesssim \norm{\nabla \widetilde \phi}_{L^2(\R^3)}+\norm{\widetilde \phi}_{L^2(\{|x|\leq R_o+3\})}+ \norm{g_o}_{H^2(\outcore)}.
        \end{multline}
        \item We localize near $\partial \incore$. This time, take $\chi =1$ for $r<(R_i+R_o)/2$ and $\chi=0$ for $r\geq (R_i+3R_o)/4$. We apply \Cref{prop:refraction} again to conclude
        \[
        \p*{\norm{\widetilde \phi}^2_{\dot H^2(\outcore\cap \{|x|< (R_o+R_i)/2\})}+\norm{\widetilde \phi}^2_{\dot H^2(\incore)}}^{1/2}\lesssim 
        \norm{\nabla \widetilde \phi}_{L^2(\R^3)}+\norm{\widetilde \phi}_{L^2(\{|x|\leq (R_o+R_i)/2\}}+ \norm{g_i}_{H^2(\outcore)}.
        \]
    \end{enumerate}
    Putting these local estimates together, we conclude that $\widetilde \phi\in \dot H^2(\incore)\times\dot H^2(\outcore)\times \dot H^2(\exterior)$ and satisfies the estimate
    \[\norm{\widetilde \phi}_{\dot H^1(\R^3)} +\norm{\widetilde \phi}_{\dot H^2(\incore)\times\dot H^2(\outcore)\times \dot H^2(\exterior)}\lesssim \norm{g_i}_{H^2(\incore)}+\norm{g_o}_{H^2(\outcore)}\lesssim \norm{f_1}_{\dot H ^1(\R^3)}. \]
    Thus
    \[\norm{f_2}_{ H^1(\incore)\times H^1(\outcore)\times  H^1(\exterior)}\lesssim \norm{f_1}_{\dot H^1(\R^3)}\]
    and therefore 
    \[\norm{f}_{\dot H^1(\incore)\times\dot H^1(\outcore)\times \dot H^1(\exterior)}\lesssim\norm{f_1}_{\dot H^1(\bb R ^3)}+\norm{f_2}_{\dot H^1(\incore)\times\dot H^1(\outcore)\times \dot H^1(\exterior)}\lesssim\norm h_{L^2(\R^3)}.\]
    It is easy to verify that $f$ solves $\tdiv f = 0$ in $L^2(\R^3)$ and $\tcurl (\mu^{-1} f) = h$ in $L^2(\R^3)$.

    It remains to show uniqueness. This follows from the fact that we can replace the original control $\norm{f_1}_{\dot H^1}\lesssim\norm{h}_{L^2}$ with $\norm{f_1}_{ H^1}\lesssim\norm{h}_{L^2}$ due to the compact support of $h$ in $\core$ (\Cref{lem:div-curl}).
    In particular, if $f$ and $g$ are two functions in $\bb B^1$ solving $\tcurl (\mu^{-1} f) = \tcurl(\mu^{-1}g) = h$, then $f-g$ satisfies 
    \[\norm {f-g}_{H^1(\incore)\times H^1(\outcore)\times  H^1(\exterior)}\lesssim 0\implies f-g=0.\]
    This concludes the proof.
\end{proof}

\begin{proof}[Proof of Lemma 2.2]
    We apply \Cref{lem:biot-savart-with-refraction} and deduce that 
    \[ \norm{ B}_{\dot{H}^1(\incore)\times \dot{H}^1(\outcore)\times \dot{H}^1(\exterior)} \lesssim \norm{\tcurl \mu^{-1} B}_{L^2(\R^3)} = \norm{\tcurl \mu^{-1} B}_{L^2(\core)}.
    \]
    The desired inequality follows. 
    \end{proof}

\begin{remark}
    We have in fact shown (for $B\in \bb B^1$) the stronger statement 
    \[\norm{\tcurl (\mu^{-1}B)}_{L^2(\R^3)}\sim \norm{B}_{\bb B^1}.\]
\end{remark}

\section{Weak solutions}\label{sec:weaksol}
In this section we define Leray-Hopf-type weak solutions for the geodynamo problem, and we prove existence of such solutions (given suitable assumptions on initial conditions, boundary conditions, and forcing terms).

We begin by deriving weak forms of each equation in order to motivate our definition of weak solutions and to highlight the role of the electric field in the exterior. Assuming there exists a piecewise smooth solution $B,(u,\omega),\theta$ to the geodynamo problem, we test $\partial_t B$ against $\mu^{-1}A$ where $A$ is a piecewise smooth test function $A$ satisfying the same transmission, $\tdiv$-free, and $\tcurl$-free conditions as $B$.
\begin{multline*}
\int_{\mathbb{R}^3} \partial_{t}B\cdot \mu^{-1}A \, \mathrm{d}x- \int_{\outcore}\operatorname{curl}(u\times B)\cdot \mu^{-1}A \,\mathrm{d}x - \int_{\incore}\operatorname{curl}(u\times B)\cdot \mu^{-1}A \,\mathrm{d}x  
\\
= \int_{\exterior} -\operatorname{curl}E\cdot \mu^{-1}A \, \mathrm{d}x -\int_{\outcore} \operatorname{curl}(\sigma^{-1} \operatorname{curl}\mu^{-1}B)\cdot \mu^{-1}A \, \mathrm{d}x- \int_{\incore} \operatorname{curl}(\sigma^{-1} \operatorname{curl}\mu^{-1}B)\cdot \mu^{-1}A \, \mathrm{d}x .
\end{multline*}
We now perform some integration by parts explicitly. First, for the frozen-in flux terms
\begin{multline*}
- \int_{\outcore}\operatorname{curl}(u\times B)\cdot \mu^{-1}A \,\mathrm{d}x - \int_{\incore}\operatorname{curl}(u\times B)\cdot \mu^{-1}A \,\mathrm{d}x  
\\
=\int_{\partial \outcore} \left[ (u\times B)\times \mu^{-1}A \right]\cdot n  \, \mathrm{d}S +\int_{\partial \incore} \left[ (u\times B)\times \mu^{-1}A \right]\cdot n  \, \mathrm{d}S 
\\
-\int_{\outcore} \operatorname{curl}\mu^{-1}A\cdot(u\times B) \, \mathrm{d}x 
-\int_{\incore} \operatorname{curl}\mu^{-1}A\cdot(u\times B) \, \mathrm{d}x .
\end{multline*}
Since $u=0$ on $\partial\core$, the boundary integral on $\partial\core$ disappears. Similarly, we can re-write the remaining boundary integral
\[
\int_{\partial\incore} [(u\times B)\times \mu^{-1}A]\cdot n\,\mathrm d S = \int_{\partial \incore }(\mu_{i}^{-1}A_{i}\cdot u_{i})(B_{i}\cdot n)-(\mu_{o}^{-1}A_{o}\cdot u_{o})(B_{o}\cdot n)  \, \mathrm{d}S
\]
and apply the transmission conditions \cref{eq:inner-BC-magnetic-normal,eq:inner-BC-magnetic-tangential,eq:inner-BC-no-slip} (using in addition the fact that $u$ is tangential on $\partial\incore$) to conclude that this term also vanishes.

Continuing with the magnetic diffusion terms, \begin{multline*}
-\int_{\outcore} \operatorname{curl}(\sigma^{-1} \operatorname{curl}\mu^{-1}B)\cdot \mu^{-1}A \, \mathrm{d}x- \int_{\incore} \operatorname{curl}(\sigma^{-1} \operatorname{curl}\mu^{-1}B)\cdot \mu^{-1}A \, \mathrm{d}x 
\\
=-\int_{\partial \outcore} (\sigma^{-1} \operatorname{curl}\mu^{-1}B\times \mu^{-1}A)\cdot n \, \mathrm{d}S 
-\int_{\partial \incore} (\sigma^{-1} \operatorname{curl}\mu^{-1}B\times \mu^{-1}A)\cdot n \, \mathrm{d}S 
\\
-\int_{\outcore} \frac{1}{\sigma} \operatorname{curl}\mu^{-1}B\cdot \operatorname{curl}\mu^{-1}A \, \mathrm{d}x 
-\int_{\incore} \frac{1}{\sigma} \operatorname{curl}\mu^{-1}B\cdot \operatorname{curl}\mu^{-1}A \, \mathrm{d}x.
\end{multline*}
Again, we apply \cref{eq:inner-BC-magnetic-tangential,eq:inner-BC-neumann} combined with the fact that $\sigma^{-1}\tcurl \mu^{-1} B\times n$ is tangential to deduce that the boundary integrals on $\partial\incore$ cancel. Finally, we account for the term involving the electric field.
\[
\int_{\exterior} \operatorname{curl}E\cdot \mu^{-1}A \, \mathrm{d}x =\int_{\partial \exterior} (E\times \mu^{-1}A)\cdot n \, \mathrm{d}S +\int_{\exterior} E\cdot \operatorname{curl}\mu^{-1}A \, \mathrm{d}x =  \int_{\partial \exterior} -(E\times n)\cdot \mu^{-1}A \, \mathrm{d}S. 
\]
We have used that $\tcurl \mu^{-1}A=0$ on $\exterior$. Using \cref{eq:outer-BC-magnetic-tangential,eq:outer-BC-neumann}, the remaining boundary integral on $\partial\exterior$ cancels with the previous boundary integral on $\partial \core$. All together, this yields
\begin{equation}\label{eq:magnetic-test}
\left\langle \partial_t B, A\right\rangle_{\bb B^0}+d(B,A) = \left\langle u\times B,\tcurl \mu^{-1} A\right\rangle_{\core}
\end{equation}
which will be the equation used to define the weak solution for the magnetic field. Observe that in the same way pressure is eliminated from the Navier-Stokes equations after testing against a divergence-free field, so is the electric field eliminated when testing against a field $A$ that is curl-free on $\exterior$.

Testing $\partial_t u$ with a $\rho v$ where $v$ is smooth and satisfying the same divergence-free constraint and boundary conditions as $u$ (except with $v=\alpha\times x$ on $\partial\incore$) produces 
\[\langle \partial_t u, \rho v\rangle + \langle u\cdot \nabla u, \rho v\rangle +\langle \nabla p, \rho v\rangle = \langle \nu\Delta u, \rho v \rangle +\langle \rho^{-1}\mu^{-1}\tcurl B\times B, \rho v\rangle +\langle \theta g \radial ,\rho v\rangle -\langle Le_3\times u, \rho v\rangle .\]
All inner products above are in $\outcore$.
After integrating by parts and applying boundary conditions this becomes
\begin{multline}\label{eq:momentum-test}
    \langle \partial_t u, \rho v\rangle_\outcore - \rho b(u,v,u) +\rho\nu a( u,v) 
    = -\mu_o^{-1}b(B,v,B) +\langle \theta g \radial ,v\rangle_\outcore -\langle Le_3\times u, \rho v\rangle_\outcore 
    \\
    + \int_{\partial \outcore} \mu_o^{-1}(B\cdot v)(B\cdot n)\d S+\rho\nu\int_{\partial \outcore} v\cdot(\nabla u) n\d S
\end{multline}
where $b(u,v,w) = \int_\outcore [(u\cdot \nabla)v]\cdot w\d x$ is the standard trilinear form associated with the Navier-Stokes equations.
In order to remove the surface integrals, we test the equation for $\omega$ with $\alpha$. We use that $u=\omega\times \radial $ and $v=\alpha\times \radial $ on $\partial \incore$.
\[ 
J\frac{\d}{\d t}\omega \cdot \alpha  = - (JLe_3\times \omega )\cdot \alpha +\int_{\partial \incore}(2\rho\nu D^s(u)+\mu_o^{-1}B_o\otimes B_o)n\cdot v\d S.
\]
Integration by parts and application of boundary conditions leads to the following.
\begin{multline}\label{eq:angular-test}
    J\frac{\d}{\d t}\omega\cdot  \alpha = 
    - (JLe_3\times \omega)\cdot \alpha \\
    +\int_{\partial \incore} \mu_0^{-1}(B\cdot n)(B\cdot v)\d S+ \rho \nu\int_{\partial \incore} v\cdot(\nabla u)n\d S-\rho\nu \int_\outcore \nabla v :\nabla u^T \d x.
\end{multline}
Combining \cref{eq:angular-test,eq:momentum-test}, canceling matching terms, and using that $2a^S(u,v) = \langle \nabla u, \nabla v\rangle + \langle \nabla u, \nabla^T v\rangle$ leads to an equation without boundary integrals.
\begin{multline}\label{eq:momentum-and-angular-test}
    \langle \partial_t (u,\omega), (v,\alpha)\rangle_{\bb U^0} - \rho b(u,v,u) +2\rho \nu a^S( u,v) 
    \\
    =  
    -\mu_o^{-1}b(B,v,B) + \langle \theta g \radial,\rho v\rangle_\outcore -\langle Le_3\times u, \rho v\rangle _{\outcore}
    -( JLe_3\times \omega)\cdot  \alpha .
\end{multline}
Testing $\partial_t \tau$
against a smooth test function $\varphi$ is relatively straightforward. After integrating by parts and applying boundary conditions the equation is 
\begin{equation}
    \label{eq:temperature-test}
    \langle \partial_t \tau, \varphi\rangle -b(u, \varphi,\tau)+\kappa a(\tau,\varphi) = -b(u,\theta^c,\varphi) -\kappa a(\theta^c,\varphi)-\kappa\langle f_b, \varphi\rangle_{\partial \incore}+\langle f, \varphi\rangle.
\end{equation}
In order to derive a formal energy identity, one sums \cref{eq:magnetic-test,eq:momentum-and-angular-test,eq:temperature-test} with $A=B$, $(v,\alpha)=(\omega,u)$ and $\varphi=\tau$. The resulting identity is found after carefully canceling terms and using boundary conditions. These detailed cancellations will be shown in the analysis of the Galerkin approximating solutions later in this section.
\begin{multline*}
    \label{eq:a-priori-energy}
    \frac{1}{2}\frac{\d}{\d t}\p*{ \norm{\mu^{-1/2}B}_{L^2(\R^3)}^2+\norm{\rho^{1/2}u}_{L^2(\outcore)}^2+ J|\omega|^2+ \norm{\tau}_{L^2(\outcore)}^2
    }
    \\
    +\left( \norm{\sigma^{-1/2}\tcurl \mu^{-1}B}_{L^2(\core)}^2 + 2\rho \nu\norm{D^S( u)}_{L^2(\outcore)}^2+\kappa \norm{\nabla \tau}_{L^2(\outcore)}^2 \right)
    \\
    =
    \langle \theta^c g\radial,\rho u\rangle_\outcore+\langle \tau g\radial , \rho u\rangle_\outcore -b(u,\theta^c,\tau)-\kappa a(\theta^c,\tau)
    \\
    - \kappa\langle f_b, \tau\rangle_{H^{-1/2}(\partial \incore),H^{1/2}(\partial\incore)}
    +\langle f, \tau\rangle_{(\bb T^1)', \bb T ^1}.
\end{multline*}
This \emph{a-priori} energy identity is indicative that the system will be amenable to Galerkin analysis, since (after further estimates) it is clear there will be at most exponential growth in time of the magnetic energy, kinetic energy, and temperature energy, as well as control over derivatives of the solution due to magnetic dissipation, viscosity in the momentum, and dissipation in the buoyancy.
We are now prepared to rigorously define a weak solution.

\begin{definition}\label{def:weak-sol}
    Let $B_0\in \bb B^0$, $(u_0,\omega_0)\in \bb U^0$, $\tau_0\in \bb T^0$. A weak solution to the \cref{eqs:outer-core,eqs:inner-core,eqs:exterior,eqs:BC-inner,eqs:BC-outer} is a a triple $(B,(u,\omega),\tau)\in L^2([0,T];\mathbb{B}^1\times\mathbb{U}^1\times \mathbb{T}^1)\cap L^\infty([0,T]; \bb B^0\times\bb U ^0\times\bb T^0)$ such that 
    for any test triple $(A, (v,\alpha), \varphi)\in 
    \cc{C}^1_0([0,T); \bb{B}^1\times\bb{U}^1\times \bb{T}^1)$, the following equations are satisfied. 
    \begin{subequations}\label{eq:weak-sol}
        \begin{equation}\label{eq:weak-magnetic}
             \int_0^T \langle B, \partial_t A\rangle_{\bb B^0} \d t - \langle B_0,A(0)\rangle_{\bb B^0}
             \\
             -\int_0^T d(B,A)  +  \langle u\times B, \tcurl \mu^{-1}A\rangle_\core \d t = 0
        \end{equation}
        \begin{multline}\label{eq:weak-momentum}
            \int_0^T \langle  (u,\omega), \partial_t (v,\alpha)\rangle_{\bb U^0}  \d t -\langle  (u,\omega)_0,  (v,\alpha)(0)\rangle_{\bb U^0}
            + \int_0^T  \rho b(u,v,u) 
            -2\rho \nu a^S( u,v)\d t 
            \\
            =  
            \int_0^T \mu_o^{-1}b(B,v,B) 
            - \langle \theta g \radial ,\rho v\rangle _\outcore
            +\langle Le_3\times u, \rho v\rangle _\outcore
            + J( Le_3\times \omega)\cdot   \alpha \d t
        \end{multline}
        \begin{multline}\label{eq:weak-temperature}
            \int_0^T \langle \tau, \partial_t \varphi\rangle_\outcore \d t -\langle \tau_0,\varphi(0)\rangle_\outcore
            + \int_0^T b(u, \varphi, \tau )
            -\kappa a(\tau,\varphi)\d t 
            \\
            =
            \int_0^T b(u,\theta^c,\varphi) 
            + \kappa a(\theta^c,\varphi)
            + \kappa\langle f_b, \varphi\rangle_{H^{-1/2}(\partial \incore),H^{1/2}(\partial\incore)}
            -\langle f, \varphi\rangle_{(\bb T^1)', \bb T ^1} \d t.
        \end{multline}
    \end{subequations}
    Moreover,
    \begin{enumerate}
        \item The solution satisfies the following energy inequality.
        \begin{multline}
            \label{eq:energy-total}
            \frac{1}{2}\p*{
            \norm{\mu^{-1/2}B}_{L^2(\R^3)}^2+\norm{\rho^{1/2}u}_{L^2(\outcore)}^2+ J|\omega|^2+ \norm{\tau}_{L^2(\outcore)}^2
            }(t)
            \\
            +\int_0^t 
            \norm{\sigma^{-1/2}\tcurl \mu^{-1}B}_{L^2(\core)}^2 + 2\rho \nu\norm{D^S( u)}_{L^2(\outcore)}^2+\kappa \norm{\nabla \tau}_{L^2(\outcore)}^2 
            \d s
            \\
            \hspace{-3.5 cm}
            \leq 
            \frac{1}{2}\p*{
            \norm{\mu^{-1/2}B}_{L^2(\R^3)}^2+\norm{\rho^{1/2}u}_{L^2(\outcore)}^2+ J|\omega|^2+ \norm{\tau}_{L^2(\outcore)}^2
            }(0)
            \\
            +
            \int_0^t
            \langle \theta^c g\radial,\rho u\rangle_\outcore+\langle \tau g\radial , \rho u\rangle_\outcore -b(u,\theta^c,\tau)-\kappa a(\theta^c,\tau)
            \\
            - \kappa\langle f_b, \tau\rangle_{H^{-1/2}(\partial \incore),H^{1/2}(\partial\incore)}
            +\langle f, \tau\rangle_{(\bb T^1)', \bb T ^1}
            \d s.
        \end{multline}
        \item The solution satisfies the initial condition.
        \begin{equation}\label{eq:initial-condition}
            B(0) = B_0, \quad (u,\omega)(0) = (u_0,\omega_0), \quad \tau(0) = \tau_0.
        \end{equation}
        \item The solution is weakly continuous with values in $\bb{B}^0\times \bb{U}^0\times \bb{T}^0$.
    \end{enumerate}
\end{definition}

\begin{remark}
    With the simplifying assumption that $\theta_b=0$, we have $\theta^c=0$ and $\theta=\tau$, so \cref{eq:weak-temperature} simplifies to 
    \begin{multline}\label{eq:weak-temperature-simplified}
        \int_0^T \langle \tau, \partial_t \varphi\rangle_\outcore \d t-\langle \tau_0, \varphi(0) \rangle_\outcore 
        + \int_0^T b(u, \varphi, \tau )
        -\kappa a(\tau,\varphi)\d t 
        \\
        =
        \int_0^T 
        \kappa\langle f_b, \varphi\rangle_{H^{-1/2}(\partial \incore),H^{1/2}(\partial\incore)}
        -\langle f, \varphi\rangle_{(\bb T^1)', \bb T ^1} \d t.
    \end{multline}
    In fact, this simplification holds if $\theta_b$ is constant, since we can take a constant extension and the terms with $\theta^c$ only depend on its gradient.
\end{remark}

\subsection{Existence}
In this section, we prove the following rigorous version of \Cref{thm:existence-intro}.

\begin{theorem}\label{thm:existence}
    Suppose $B_0\in \bb B^0$, $(u_0,\omega_0)\in \bb U^0$, $\tau_0\in \bb T^0$, $\theta_b\in H^{1/2}(\partial\core)$, $f_b\in H^{-1/2}(\partial \core)$, and $f\in L^2_t(\bb T^1)'$. Then there exists a triple $(B,(u,\omega),\tau)$ that satisfies \Cref{def:weak-sol}, i.e., a weak solution to the geodynamo problem.
\end{theorem}

Our overall strategy will be to use Galerkin approximation. We will first find a sequence of approximate solutions to finite dimensional approximating problems, and then we will show that there exists a limit of that sequence satisfying \Cref{def:weak-sol}. The functional framework is designed to be amenable to Galerkin approximation, in that an energy inequality will be satisfied at the level of the approximations, and in that relevant embeddings of function spaces will be available to deduce convergence to a true solution. In particular, we will need to use the Korn and Biot-Savart-type inequalities (\Cref{lem:curl-norm,lem:poincare-korn}) to deal with the nonlinear terms. We will also need to use the Aubin-Lions lemma to show strong convergence on a subsequence of the approximations in $L^2_tL^2_x$.

\subsubsection{Galerkin approximation}
Because $\mathbb{B}^1, \mathbb{U}^1$, and $\mathbb{T}^1$ are separable Hilbert spaces, they each have a countable orthogonal basis, which we enumerate as $\{A_j\}, \{(v,\alpha)_j\}$, and $\{\varphi_j\}$, respectively. We let $\mathcal{B}_m,\mathcal{U}_m$, and $\mathcal{T}_m$ 
be the spans of the first $m$ basis elements of each space, respectively. We also overload $P_m$ to be defined as orthogonal projection in the appropriate $L^2$-like space (e.g., $\bb B ^0$ for $B$, $\bb U^0$ for $(u,\omega)$) onto the span of the first $m$ basis elements of whichever space is relevant. We then define the $m$th approximate solution $(B_m, (u,\omega)_m, \tau_m)$ as 
\begin{subequations}
    \begin{equation}
        B_m = \sum_{j=1}^m g^B_{j,m}(t)A_j
    \end{equation}
    \begin{equation}
        (u,\omega)_m = \sum_{j=1}^m g^u_{j,m}(t)(v,\alpha)_j
    \end{equation}
    \begin{equation}
        \tau_m = \sum_{j=1}^m g^\tau_{j,m}(t)\varphi_j
    \end{equation}
\end{subequations}
where the $g_{j,m}^B(t), g_{j,m}^u(t), g_{j,m}^\tau(t)$ are chosen such that the following equations are satisfied $\forall j\in\{1\,\ldots,m\}$:
\begin{subequations}\label{eqs:galerkin}
    \begin{equation}\label{eq:magnetic-galerkin}
        \langle \partial_t B_m, \mu^{-1}A_j\rangle_{\R^3}  + d(B_m,A_j) 
        = \langle u_m\times B_m, \tcurl \mu^{-1}A_j\rangle_\core,
    \end{equation}
    \begin{multline}
        \label{eq:momentum-and-angular-galerkin}
        \vspace{-6 pt}
        \langle \partial_t u_m, \rho v_j\rangle_\outcore 
        + J \frac{d}{dt}\omega_m\cdot   \alpha _j
        - \rho b(u_m,v_j,u_m) 
        +2\rho\nu a^S( u_m,v_j) 
        \\
        =  
        -\mu_o^{-1} b(B_m,v_j,B_m) 
        + \langle \theta_m g \radial,\rho v_j\rangle_\outcore 
        -\langle Le_3\times u_m, \rho v_j\rangle_\outcore  
        - J L(e_3\times \omega_m)\cdot  \alpha_j,
    \end{multline}
    \begin{multline}
        \label{eq:temperature-galerkin}
        \vspace{-6pt}
        \langle \partial_t \tau_m, \varphi_j \rangle_\outcore 
        -b (u_m, \varphi_j,\tau_m)
        +\kappa a(\tau_m,\varphi_j) 
        \\ 
        = 
        -b(u_m,\theta^c,\varphi_j) 
        -\kappa a(\theta^c,\varphi_j)
        -\kappa\langle f_b, \varphi_j \rangle_{H^{-1/2}(\partial \incore),H^{1/2}(\partial \incore)}
        +\langle f, \varphi_j\rangle_{(\bb T^1)', \bb T ^1}.
    \end{multline}
\end{subequations}
We supplement this system of \glspl{ode} with the initial conditions 
\[(B_m(0),(u,\omega)_m(0), \tau_m(0)) = P_m(B_0,(u,\omega)_0,\tau_0).\]

Since this is a locally Lipschitz system of \glspl{ode}, we know that there is a solution up to some time $T^\ast$ for each $m$. Our next task is to show that $T^\ast$ is uniform in $m$, for which we need energy estimates that are uniform in $m$.

\subsubsection{Galerkin energy}
The energy of an approximate solution $(B_m, (u,\omega)_m, \tau_m)$ is given by testing against $(\mu^{-1}B_m, \rho (u,\omega)_m, \tau_m)$ according to
\cref{eq:magnetic-galerkin,eq:momentum-and-angular-galerkin,eq:temperature-galerkin} and summing them all together. 
\begin{align*} 
    &\frac{d}{dt}\frac{1}{2} \norm{\mu^{-1/2}B_m}_{L^2(\R^3)}^2+\norm{\sigma^{-1/2} \tcurl \mu^{-1} B_m}_{L^2(\outcore)}^2
    \\
    & \hspace{2cm} + \frac{d}{dt}\frac{1}{2}\norm{\rho^{1/2}u_m}_{L^2(\outcore)}^2
    + J\frac{d}{dt}\frac{1}{2}|\omega_m|^2
    + 2\rho \nu a^S(u_m,u_m)
    + \frac{d}{dt}\frac{1}{2}\norm{\tau_m}_{L^2(\outcore)}^2+\kappa a(\tau_m,\tau_m)
    \\
    & \hspace{ 1.2 cm} = 
    \underbrace{\langle u_m\times B_m, \tcurl \mu^{-1}B_m\rangle_\core }_{\text{(c)}}
    +\underbrace{\rho b(u_m,u_m,u_m)}_{\text{(a)}} - \underbrace{\mu_o^{-1}b(B_m,u_m,B_m)}_{\text{(c)}}
    \\
    &\hspace{1.6cm}
     + \langle \theta_m g \radial, \rho u_m\rangle _\outcore  
     - \underbrace{\langle Le_3\times u_m, \rho u_m\rangle_\outcore }_{\text{(b)}} 
     - \underbrace{J( Le_3\times  \omega_m)\cdot ( \omega_m)}_{\text{(b)}} 
     + \underbrace{b(u_m,\tau_m,\tau_m)}_{\text{(a)}}
    \\
    &\hspace{1.6cm} 
    -b(u_m,\theta^c,\tau_m)
    -\kappa a(\theta^c, \tau_m)
    -\kappa \langle f_b,\tau_m\rangle_{H^{-1/2}(\partial \incore),H^{1/2}(\partial \incore)} 
    +\langle f, \tau_m\rangle _{(\bb T^1)', \bb T^1}.
\end{align*}
We make the following simplifications.
\begin{itemize}
    \item We use that $b(u_m,\cdot, \cdot)$ is antisymmetric in the last two arguments since $u_m\in H^1_{\tdiv}$, $u_m=0$ on $\partial \core$, and $u_m\cdot n = 0$ on $\partial \incore$. This eliminates the terms labeled (a).
    \item We use that $\langle a\times b, a\rangle=0$ to eliminate terms labeled (b).
    \item The last simplification is to show $\langle u_m\times B_m,\tcurl \mu^{-1} B_m\rangle_\core$ and $ \mu^{-1}b(B_m, u_m, B_m)$ (labeled (c) ) cancel each other out. We first integrate by parts (and use the boundary conditions from the definition of $\mathbb{B}^1$) to deduce 
    \[
    \langle u_m\times B_m,\tcurl\mu^{-1} B_m\rangle_\core = \langle \tcurl(u_m\times B_m) , \mu^{-1}B_m\rangle_\incore +\langle \tcurl(u_m\times B_m) , \mu^{-1}B_m\rangle_\outcore.
    \]
    Next we expand
    \[
    \tcurl(u_m\times B_m) = u_m\tdiv B_m - B_m\tdiv u_m + (B_m\cdot \nabla) u_m - (u_m\cdot \nabla) B_m.
    \] 
    The first two terms vanish due to incompressibility. The last also vanishes when we consider the $L^2$ inner product over $\outcore$ or $\incore$ with $\mu^{-1} B_m$ due to antisymmetry. We are left with 
    \[\langle u_m\times B_m,\tcurl\mu^{-1} B_m\rangle_\core = \mu_o^{-1} \left\langle (B_m\cdot\nabla)u_m,B_m\right\rangle_\outcore + \mu_i^{-1} \left\langle (B_m\cdot\nabla)u_m,B_m\right\rangle_\incore.\] 
    At this point we must show that $\left\langle (B_m\cdot\nabla)u_m,B_m\right\rangle_\incore=0$. To do so, we use that $u_m=\omega_m\times x$ on $\incore$:
    \[\int_\incore(B_m\cdot \nabla)\p*{\omega\times x}\cdot B_m\d x = \int_\incore\p*{\omega\times B_m}\cdot B_m\d x = 0.\]
\end{itemize}
We now have the following energy expression.
\begin{multline}\label{eq:galerkin-energy-exact}
    \frac{d}{dt}\frac{1}{2}\p*{
     \norm{\mu^{-1/2}B_m}_{L^2(\R^3)}^2
    + \rho\norm{u_m}_{L^2(\outcore)}^2
    + J |\omega_m|^2+ \norm{\tau_m}_{L^2(\outcore)}^2
    }
    \\
    + \norm{\sigma^{-1}\tcurl \mu^{-1}B_m}_{L^2(\core)}^2 + 2\rho \nu\norm{ D^S(u_m)}_{L^2(\outcore)}^2+\kappa \norm{\nabla \tau_m}_{L^2(\outcore)}^2
    \\
    =
    \langle \theta^c g\radial,\rho u_m\rangle_\outcore 
    +\langle \tau_m g\radial , \rho u_m\rangle_\outcore  
    -b(u_m,\theta^c,\tau_m)
    \\
    -\kappa a(\theta^c,\tau_m)
    -\kappa \langle f_b,\tau_m\rangle_{H^{-1/2}(\partial\incore),H^{1/2}(\partial \incore)}
    +\langle f, \tau_m\rangle_{(\bb T^1)', \bb T^1}.
\end{multline}

Finally we need to estimate the right hand side of \cref{eq:galerkin-energy-exact}. Recall that $\theta^c$ is an extension into $\outcore$ of the boundary conditions $\theta=\theta_b$ on $\partial\core$ and $\nabla\theta\cdot n=0$ on $\partial \incore$. In particular, we can pick this extension to satisfy $\norm{\nabla\theta^c}_{L^2(\outcore)}\lesssim \norm{\theta_b}_{H^{1/2}(\partial\outcore)}$. We establish the following crude bounds (note that some terms can be estimated in multiple ways by modifying exponents and using different Sobolev embeddings to have slightly different requirements on the integrability of $\theta_b, f_b, f$): 
\begin{itemize}
    \item $|\langle \theta^c g\radial, \rho u_m\rangle_\outcore |\lesssim \norm{\rho u_m}_{L^2(\outcore)}\norm{\theta^c}_{L^2(\outcore)}\norm{g \radial }_{L^\infty(\outcore )} \lesssim_{\rho, g,\theta_b} \norm{u_m}_{L^2(\outcore)}$.
    \item $|\langle \tau_m g\radial, \rho u_m\rangle_\outcore |\lesssim \norm{\tau_m}_{L^2(\outcore)}\norm{\rho u_m}_{L^2(\outcore)}\norm{g \radial }_{L^\infty(\outcore)}\lesssim_{\rho,g} \norm{\tau_m}_{L^2(\outcore)}\norm{u_m}_{L^2(\outcore)} $.
    \item Using Ladyzhenskaya's inequality,
    \begin{multline*} |b(u_m,\theta^c,\tau_m)|\leq \norm{u_m}_{L^4(\outcore) }\norm{\nabla \theta^c}_{L^2(\outcore )}\norm{\tau_m}_{L^4(\outcore)} 
    \\
    \lesssim_{\theta_b} \norm{u_m}_{L^2(\outcore)}^{1-d/4} \norm{\tau_m}_{L^2(\outcore)}^{1-d/4} \norm{\nabla u_m}_{L^2(\outcore)}^{d/4} \norm{\nabla \tau_m}_{L^2(\outcore)}^{d/4}. 
    \end{multline*}
    \item $|\kappa a(\theta^c, \tau_m)| \lesssim  \kappa \norm{\nabla \theta^c}_{L^2(\outcore)} \norm{\nabla \tau_m}_{L^2(\outcore)} \lesssim_{\kappa,\theta_b} \norm{\nabla \tau_m}_{L^2(\outcore)}$. 
    \item Using a trace inequality for $\tau_m$, \[
    |\kappa \langle f_b,\tau_m\rangle_{H^{-1/2}(\partial\incore),H^{1/2}(\partial \incore) }|
    \lesssim \kappa \norm{f_b}_{H^{-1/2}(\partial \incore )}\norm{\tau_m }_{H^{1/2}(\partial{\incore })}
    \lesssim_{\kappa,f_b}  \norm{\nabla \tau_m }_{L^2(\outcore)}.\]
    \item $|\langle f,\tau_m\rangle_{(\bb T^1)', \bb T^1} |\lesssim \norm{f}_{(\bb T^1)'}\norm{\tau_m}_{H^1(\outcore)}
    \lesssim_{f}\norm{\nabla\tau_m}_{L^2(\outcore)}$.
\end{itemize}
Using these estimates, it follows that
\begin{multline}\label{eq:galerkin-energyest-simple}
    \frac{d}{dt} F_m
    + G_m
    \lesssim_{\rho,\kappa,g,\theta_b,f_b, f}
    \norm{u_m}_{L^2(\outcore)} + \norm{\tau_m}_{L^2(\outcore) }\norm{u_m}_{L^2(\outcore)}+ \norm{\nabla \tau_m}_{L^2(\outcore)} + \norm{\tau_m}_{L^2(\outcore)} 
    \\
    + \norm{u_m}_{L^2(\outcore)}^{1-d/4}\norm{\tau_m}_{L^2(\outcore)}^{1-d/4}\norm{\nabla u_m}_{L^2(\outcore)}^{d/4}\norm{\nabla \tau_m}_{L^2(\outcore)}^{d/4}
\end{multline}
where $F_m$ and $G_m$ are defined as the corresponding quantities in \cref{eq:galerkin-energy-exact}, i.e.,
\[F_m = \frac{1}{2}\p*{
     \norm{\mu^{-1/2}B_m}_{L^2(\R^3)}^2
    + \rho\norm{u_m}_{L^2(\outcore)}^2
    + J |\omega_m|^2+ \norm{\tau_m}_{L^2(\outcore)}^2
    },
\]
\[G_m = 
     \norm{\sigma^{-1}\tcurl \mu^{-1}B_m}_{L^2(\core)}^2 + 2\rho \nu\norm{ D^S( u_m)}_{L^2(\outcore)}^2+\kappa \norm{\nabla \tau_m}_{L^2(\outcore)}^2.
\]
By applying Young's inequality and \Cref{lem:poincare-korn}, we absorb $\norm{\nabla \tau_m}_{L^2(\outcore)}$ and $\norm{\nabla u_m}_{L^2(\outcore)}$ into $G_m$. We also use Young's to convert other terms from the right side into terms that appear in $F_m$. In the most simplified form,
\begin{equation}\label{eq:galerkin-energyest-crude}
    \frac{d}{dt}F_m+ \frac{1}{2}G_m\lesssim_{\rho,\nu,\kappa,g,\theta_b,f_b, f} 1 + F_m.
\end{equation}
Because $G_m\geq 0$, we can apply Gr\"{o}nwall's inequality to 
\[\frac{d}{dt}F_m \leq C+CF_m\]
which leads to 
\[F_m(t)\leq (F_m(0)+1)\exp(Ct).
\]
Finally, we use this estimate for $F_m$ and integrate both sides of \cref{eq:galerkin-energyest-crude} to arrive at 
\[
F_m(t)+\int_0^t \frac{1}{2}G_m(s)\d s\leq CT + C(F_m(0)+1)(1+\exp(CT))\leq CT+C(F(0)+1)(1+\exp(CT)).
\]
Here we define 
\[F(0) = \frac{1}{2}\p*{\norm{\mu^{-1/2}B_0}_{L^2(\R^3)}^2 + \rho\norm{u_0}_{L^2(\outcore)}^2 + J |\omega_0|^2 + \norm{\tau_0}_{L^2(\outcore)}^2} \]
which satisfies
\[
    F_m(0) = \frac{1}{2}\p*{ \norm{\mu^{-1/2}B_m(0)}_{L^2(\R^3)}^2 + \rho\norm{u_m(0)}_{L^2(\outcore)}^2 +  J|\omega_m(0)|^2 + \norm{\tau_m(0)}_{L^2(\outcore)}^2
    } 
    \leq F(0)
\]
because the $m$th initial conditions are defined by orthogonal projection in $L^2$ onto the subspace spanned by the first $m$ basis elements.

\begin{remark}
    If we reduce to the simpler case $\theta_b=0$ (or $\theta_b$ is constant) then we can derive that the total energy increases at most linearly.
    First, we consider \cref{eq:temperature-galerkin}, and we take $\tau_m$ as the test function
    \[
    \frac{d}{dt}\frac{1}{2}\norm{\tau_m}_{L^2(\outcore)}^2+\kappa\norm{\nabla\tau_m}_{L^2(\outcore)}^2\leq \kappa\norm{f_b}_{H^{-1/2}(\partial \incore)}\norm{\tau_m}_{H^{1/2}(\partial \incore)}+\norm{f}_{(\bb T^1)'}\norm{\tau_m}_{H^1(\outcore)}.
    \]
    Now we apply trace, Poincar\'{e}'s, and Young's inequalities to deduce the following form
    \[
    \frac{d}{dt}\frac{1}{2}\norm{\tau_m}_{L^2(\outcore)}^2 + \kappa \norm{\nabla\tau_m}_{L^2(\outcore)}^2 \leq C_1 \norm{f_b}^2_{H^{-1/2}(\partial \incore)} + \frac{\kappa}{4} \norm{\nabla \tau_m}_{L^2(\outcore)}^2 + C_2 \norm{f}_{(\bb T)'}^2 + \frac{\kappa}{4} \norm{\nabla \tau_m}_{L^2(\outcore)}^2.
    \]
    Finally, we absorb $\frac{\kappa}{2}\norm{\nabla \tau_m}_{L^2}^2$ in the left hand side and apply Poincar\'{e} again,
    \[
    \frac{d}{dt}\norm{\tau_m}_{L^2(\outcore)}^2 +\norm{\tau_m}_{L^2(\outcore)}^2\lesssim_{\kappa} \norm{f_b}_{H^{-1/2}(\partial \incore)}^2+\norm{f}_{(\bb T^1)'}^2.
    \]
    The general form of this differential inequality is $\frac{d}{dt}x+x\leq C$. One may use integrating factors to deduce that $\norm{\tau_m}_{L^2}$ is bounded uniformly in time. 
    Using that $\norm{\tau_m}_{L^2}\lesssim 1$, as well as the fact that $b(u_m,\theta^c,\tau_m) = 0$ (the derivative falls on the function $\theta^c$, which is constant),
    we can improve \cref{eq:galerkin-energyest-simple} to 
    \begin{equation}
        \frac{d}{dt} F_m + G_m
        \lesssim_{\rho ,\kappa,g,\theta_b,f_b, f}
        \norm{u_m}_{L^2(\outcore)} + \norm{\nabla \tau_m}_{L^2(\outcore)} .
    \end{equation}
    Applying Young's inequality, Poincar\'{e}'s inequality, and \Cref{lem:poincare-korn} now leads to
    \[
        \frac{d}{dt} F_m + \frac{1}{2}G_m
        \lesssim_{\rho ,\nu,\kappa,g,\theta_b,f_b, f} 1 .
    \]
   in~\eqref{eq:galerkin-energyest-simple} which implies a linear bound on growth in time of the total energy $F_m$.
\end{remark}

The Galerkin energy estimate \cref{eq:galerkin-energyest-crude} demonstrates that the energy of the \gls{ode} solutions are uniformly bounded in $m$, so the sequence of approximations $\{(B_m,(u,\omega)_m,\tau_m)\}$ exist on the same interval $[0,T^\ast]$. Moreover, since the energy is finite at all finite times, we derive global solutions to each approximating \gls{ode}, i.e., $(B_m,(u,\omega)_m,\tau_m)$ exist on $[0,\infty)$.

\subsubsection{Convergence to a weak solution}
The next step is to derive a limit solution from our sequence of approximate solutions. For any $T>0$ and each $m>0$, we have solutions $(B_m, (u_m,\omega_m), \tau_m)$ to the approximating \gls{ode}. Moreover, $B_m, u_m,$ and $\tau_m$ are bounded in $L^\infty_tL^2_x$, $\tcurl\mu^{-1} B_m, D^Su_m, \nabla \tau_m$ in  $L^2_tL^2_x$, while $\omega_m$ is bounded in $L_t^\infty\R^3_x$, all uniformly in $m$. Applying Banach-Alaoglu and the equivalence of weak and weak-* topologies for reflexive Banach spaces, we deduce that on a subsequence (not relabeled)
\begin{align*}
    B_m,u_m,\tau_m &\xrightharpoonup[L^\infty_tL^2_x]{\ast} B,u,\tau 
    \\
    \tcurl \mu^{-1} B_m,D^S u_m,\nabla \tau_m &\xrightharpoonup[L^2_t L^2_x]{} \tcurl \mu^{-1} B,D^S u,\nabla \tau 
    \\
    \omega_m &\xrightharpoonup[L^\infty_t]{\ast} \omega.
\end{align*}
In fact, using that $L^\infty_tL^2_x\subset L^2_tL^2_x$ (for finite times) and using the Poincar\'{e} inequality and embeddings from \Cref{lem:curl-norm,lem:poincare-korn}, we have that $B_m, u_m,$ and $\tau_m$ are uniformly bounded in $L^2_t\bb B^1, L^2_t\bb U^1,L^2_t\bb T^1$ respectively, giving the additional weak convergence
\[B_m,u_m,\tau_m \xrightharpoonup[L^2_t\bb X]{} B,u,\tau\]
where $\bb X$ is the appropriate space.

We now show that $(B,(u,\omega),\tau)$ satisfy \Cref{def:weak-sol}. Throughout the rest of this section, we will take $\chi\in \cc C_0^1([0,T))$ to be an arbitrary smooth temporal function.

Let's begin by showing that \Cref{eq:weak-magnetic} is satisfied. 
We will assume our test functions are tensors of a spatial function and a smooth temporal function $\chi$.
Specifically, multiply \cref{eq:magnetic-galerkin} by $\chi(t)$. Then, for any $j\leq m$
we have 
\[
\int_0^T \langle \partial_t B_m, \chi \mu^{-1} A_j\rangle_{\R^3}  +\chi\langle \sigma^{-1}\tcurl \mu^{-1} B_m, \tcurl \mu^{-1}A_j\rangle_\core - \chi\langle u_m\times B_m, \tcurl \mu^{-1} A_j\rangle_\core\d t  = 0.
\]
Integrating by parts in time yields
\begin{multline*}
    \underbrace{\langle B_m(0), \chi(0)\mu^{-1}A_j\rangle_{\R^3}}_{(\text{a})}-\int_0^T \chi'\underbrace{\langle B_m,  \mu^{-1} A_j\rangle_{\R^3}}_{(\text{b})}  \d t
    \\
    +\int_0^T \chi\underbrace{\langle \sigma^{-1}\tcurl \mu^{-1} B_m, \tcurl \mu^{-1}A_j\rangle_\core}_{(\text{c})}- \chi\underbrace{\langle u_m\times B_m, \tcurl \mu^{-1}A_j\rangle_\core}_{(\text{d})}\d t  = 0.
\end{multline*}
As we send $m\to\infty$, it is clear that the linear terms (a), (b), and (c) converge to
\[
\langle B_0, \chi(0)\mu^{-1}A_j\rangle_{\R^3}-\int_0^T \chi'\langle B, \mu^{-1} A_j\rangle_{\R^3}+\chi \langle \sigma^{-1}\tcurl \mu^{-1} B, \tcurl \mu^{-1}A_j\rangle_\core\d t
\]
due to the weak convergence of $B_m$ and $\tcurl \mu^{-1} B_m$ in $L^2_tL^2_x$.

For term (d), we require convergence of $u_m$ or $B_m$ in $L^2_tL^2_x$ (along a further subsequence, again without relabeling).
\begin{lemma}[restate=lem:compactness]
    \label{lem:compactness}
    The sequence $\{u_m\}$ (where $u_m$ in this case are extended to $\incore$ according to \cref{eq:momentum-inner-core}) converges strongly in $L^2_tL^2(\core)$. Similarly, the sequence $\{B_m\}$ converges strongly in $L^2_tL^2(\core)$. 
\end{lemma}
The proof follows standard techniques using fractional derivatives and fractional Aubin-Lions lemma. For this reason, the proof is included in the appendix (\Cref{sec:appendix-compactness}).

We compute 
\begin{multline*}
    \int_0^T \int_\core u_m\times B_m \cdot \chi\tcurl \mu^{-1}A_j\d x \d t =
    \\
    \int_0^T\int_\core  (u_m-u)\times B_m\cdot \chi\tcurl \mu^{-1} A_j\d x\d t+\int_0^T\int_\core  (\chi\tcurl \mu^{-1}A_j\times u)\cdot B_m\d x\d t.  
\end{multline*}
Observe that $\chi\tcurl\mu^{-1} A_j\times u\in L^2_t (\bb B^1)'$ by picking some $\psi\in L^2_t\bb B^1$ and estimating
\begin{align*}
    \abs*{\int_0^T\int_\core (\chi\tcurl\mu^{-1} A_j\times u)\cdot \psi\d x\d t}
    &\lesssim \norm{\chi \tcurl \mu^{-1} A_j}_{L^\infty_tL^2_x}\norm{u}_{L^2_tH^1_x}\norm{\psi}_{L^2_t(H^1(\incore)\times H^1(\outcore))}
    \\
    &\lesssim \norm{\psi}_{L^2_t\bb B^1}
\end{align*}
where we used the embedding \Cref{lem:curl-norm}.
Due to the weak convergence $B_m\xrightharpoonup[L^2_t\bb B^1]{}B$ we have 
\[
\int_0^T\int_\core  (\chi\tcurl \mu^{-1}A_j\times u)\cdot B_m\d x\d t\to \int_0^T\int_\core  (\chi\tcurl \mu^{-1} A_j\times u)\cdot B\d x\d t.
\]
On the other hand, we can estimate 
\[
\abs*{\int_0^T\int_\core (u_m-u)\times B_m\cdot \chi\tcurl \mu^{-1} A_j\d x\d t}
\leq
\int_0^T\norm{u_m-u}_{L^4(\core)}\norm{B_m}_{L^4(\core)}\norm{\chi\tcurl\mu^{-1} A_j}_{L^2(\core)}\d t
\]
to which we can apply Ladyzhenskaya's inequality and strong convergence of $u_m$ in $L^2_tL^2(\core)$ to deduce
\begin{multline*}
    \int_0^T \norm{u_m-u}_{L^2(\core)}^{1-d/4} \norm{u-u_m}_{H^1(\core)}^{d/4} \norm{B_m}_{H^1(\incore)\times H^1(\outcore)} \norm{\chi\tcurl \mu^{-1} A_j}_{L^2(\core)}\d t
    \\
    \leq 
    \norm{u_m-u}_{L^2_tL^2_x}^{1-d/4}\norm{u-u_m}_{L^2_tH^1_x}^{d/4}\norm{B_m}_{L^2_t\bb B ^1}\norm{\chi\tcurl \mu^{-1} A_j}_{L^\infty_t L^2(\core)}
    \lesssim \norm{u_m-u}_{L^2_tL^2_x}^{1-d/4}\to 0.
\end{multline*}

We can immediately generalize to the case where $A$ is a finite linear combination of $A_j$
\[
\langle B, \chi(0)\mu^{-1} A\rangle_{\R^3}-\int_0^T {\langle B, \chi' A\rangle_{\R^3}}  +\chi\langle \sigma^{-1}\tcurl \mu^{-1} B, \tcurl\mu ^{-1} A\rangle_\core  - \chi{\langle u\times B, \tcurl \mu^{-1} A\rangle_\core}\d t  = 0
\]
and then general $A\in \bb B^1$ follows from density and boundedness of each term above as a functional on $\bb B^1$. This completes the proof that \cref{eq:weak-magnetic} is satisfied by $(B,(u,\omega),\tau)$ for any test function in $\cc C^1_0([0,T);\bb B^1)$.

Next, let us check that \cref{eq:weak-momentum} is satisfied by our limit. This time we multiply \cref{eq:momentum-and-angular-galerkin} with $\chi(t)$ for arbitrary $\chi\in \cc C_0^1([0,T))$.
\begin{multline*}
    \int_0^T \langle \partial_t u_m, \chi\rho  v_j\rangle_\outcore + \chi J\frac{d}{dt}\omega_m\cdot   \alpha_j+ \chi \rho b(u_m,u_m,v_j) +2\rho \nu  a^S( u_m,v_j)  \d t
    \\
    =  
    \int_0^T -\mu_o^{-1}\chi b(B_m,v_j,B_m) + \chi \langle \theta_m g \radial,\rho v_j\rangle_\outcore -\chi\langle Le_3\times u_m, \rho v_j\rangle _\outcore 
    -\chi J( Le_3\times \omega_m)\cdot \alpha_j  \d t.
\end{multline*}
As before, we integrate the time derivative by parts. All the linear terms converge using weak convergence of $u$ and weak-star convergence of $\omega$. The nonlinear terms converge according to the same argument as the nonlinear terms for the magnetic field equation, this time requiring strong convergence of $u_m$ and $B_m$ in $L^2_tL^2(\core )$.
For example,
\[
\int_0^T -\mu_o^{-1}\chi b(B_m,v_j,B_m)\d t = \int_0^T -\mu_o^{-1}\chi b(B_m-B,v_j,B_m)\d t + \int_0^T -\mu_o^{-1}\chi b(B,v_j,B_m)\d t
\]
so that we estimate 
\begin{multline*}
    \abs*{\int_0^T -\mu_o^{-1}\chi b(B_m-B,v_j,B_m)\d t}\leq\int_0^t \norm{B_m-B}_{L^4(\outcore)}\norm{\chi \nabla v_j}_{L^2(\outcore)}\norm{B_m}_{L^4(\outcore)}\d t
    \\
    \leq \norm{B_m-B}^{1-d/4}_{L^2_tL^2(\outcore)}\norm{B-B_m}^{d/4}_{L^2_tH^1(\outcore)}\norm{B_m}_{L^2_t\bb B^1}\norm{\chi \nabla v_j}_{L^\infty_tL^2(\outcore)}\leq C\norm{B_m-B}^{1-d/4}_{L^2_tL^2(\outcore)}\to 0.
\end{multline*}

The same density argument works to show that the equation then holds for any test function in $\cc C^1_0([0,T);\bb U^1)$. We conclude that \cref{eq:weak-momentum} is also satisfied by $(B,(u,\omega), \tau)$.

Finally, we test \cref{eq:temperature-galerkin} against $\chi \varphi_j$.
\begin{multline*}
    \int_0^T\langle \partial_t \tau_m, \chi \varphi_j\rangle_\outcore  +\chi b(u_m,\tau_m, \varphi_j)+\kappa \chi a(\tau_m, \varphi_j)\d t 
    \\
    = \int_0^T-\chi b(u_m,\theta^c,\varphi_j) -\kappa \chi a(\theta^c,\varphi_j)-\chi \kappa\langle f_b, \varphi_j\rangle_{H^{-1/2}(\partial\incore), H^{1/2}(\partial \incore)} + \chi\langle f, \varphi_j\rangle_{(\bb T^1)', \bb T^1}\d t.
\end{multline*}
Again, integration by parts, convergence of linear and nonlinear terms, and generalizing to any test function in $\cc C^1_0([0,T);\bb T^1)$ all work as before. 
We conclude that $(B,(u,\omega),\tau)$ satisfies \cref{eq:weak-temperature}.
\begin{remark}
    Although $\tau_m$ does convergence strongly in $L^2_tL^2(\outcore)$, we do not actually need to use that convergence; only the strong convergence of $u_m$ is needed.
    To see this, observe that 
    \[
    \int_0^T\chi b(u_m,\tau_m,\varphi_j )\d t = -\int_0^T \chi b(u_m,\varphi_j,\tau_m)\d t = -\int_0^T \chi b(u,\varphi_j,\tau_m) -\chi b(u-u_m,\varphi_j,\tau_m)\d t.
    \]
    The first term converges due to weak convergence of $\tau_m$ in $L^2_tH^1(\outcore)$ (and that $u\cdot\nabla \varphi_j\in L^2_t(\bb T^1)'$). We can estimate the second term by 
    \begin{equation}\label{eq:nonlin-ladyzhenskaya}
        |b(u-u_m,\varphi_j, \tau_m)|\leq \norm{u-u_m}_{L^2(\outcore)}^{d/4} \norm{u-u_m}_{H^1(\outcore)}^{1-d/4} \norm{\nabla\varphi_j}_{L^2(\outcore)} \norm{\tau_m}_{H^1(\outcore)}
    \end{equation}
    which, after integrating in time, converges to 0 due to strong convergence of $u$ in $L^2_tL^2_x$ and boundedness of the other terms.
\end{remark}

We have proven the following lemma.
\begin{lemma}\label{lem:existence-weakeqs}
    Suppose $B_0\in \bb B^0$, $(u_0,\omega_0)\in \bb U^0$, $\tau_0\in \bb T^0$, $\theta_b\in H^{1/2}(\partial \core )$, $f_b\in H^{-1/2}(\partial \core)$, and $f\in L^2_tH^{-1}(\outcore)$. Then there exists $(B,(u,\omega),\tau)$ (in the function space specified in \Cref{def:weak-sol}) that satisfy \cref{eq:weak-magnetic,eq:weak-momentum,eq:weak-temperature}.
\end{lemma}

\subsubsection{Initial condition}
Now that we have existence, we need to sort out the remaining properties of a weak solution (\Cref{def:weak-sol}). We start with the initial condition.
\begin{lemma}\label{lem:IC-limit}
    Let $(B, (u,\omega), \tau)$ be the functions from \Cref{lem:existence-weakeqs}. Then they satisfy \cref{eq:initial-condition}.
\end{lemma}
\begin{proof}
We know $(B, (u,\omega), \tau)$ satisfy \cref{eq:weak-magnetic,eq:weak-momentum,eq:weak-temperature}. We will only explicitly show how to argue for the convergence of $B$ and $(u,\omega)$, but the same idea also works for $\tau$. We can re-package the magnetic field equation \cref{eq:weak-magnetic} as an equation of the form
\begin{equation}
    \label{eq:distribution-time-magnetic}
    \frac{\d}{\d t}\langle B, \mu^{-1} A\rangle_{L^2(\bb R^3)} = \frac{\d}{\d t}\langle B, \mu^{-1} A\rangle_{(\bb B^1)',\bb B^1}  = \langle f^B,  A\rangle_{(\bb B^1)',\bb B^1}
\end{equation}
(in the distributional sense in time) for any $A\in \bb B^1$. We define $f^B$ by 
\[
f^B = \cc{N}(u,B)-\cc{D}B
\]
where $\cc{N}:H^1(\core)\times \bb B^1\to (\bb B^1)'$ is defined by $\langle \cc{N}(u,B),A\rangle = \int_\core (u\times B)(\tcurl \mu^{-1} A)\d x$
and $\cc{D}:\bb B^1\to (\bb B^1)'$ is defined by $\langle \cc{D}B,A\rangle = \int_\core \sigma^{-1}\tcurl \mu^{-1} B\cdot \tcurl \mu^{-1} A\d x$. 
We verify that $f^B\in L^1_t(\bb B^1)'$.
First, we estimate the frozen-in flux term.
\begin{multline}\label{eq:frozen-in-flux-est}
    \| \mathcal{N}(B,u) \|_{ (\bb B^1)' } 
    =
    \sup_{\| A \|_{ \bb B^1}=1} \left\lvert \int_{\core} (u\times B)\cdot \operatorname{curl}\mu^{-1}A \, \mathrm{d}x  \right\rvert 
    \\
    \leq\sup_{\| A \|_{ \bb B^1}=1 } \| u \|_{ L^4(\core) } \| B \|_{ L^4(\core) } \| A \|_{ H^1(\core) }
    \lesssim \| u \|_{ H^1(\core) } \| B \|_{ H^1(\core) }.
\end{multline}
Next we estimate the magnetic diffusion term.
\begin{multline}\label{eq:magnetic-diffusion-est}
    \| \mathcal{D}B \|_{ (\bb B^1)' } =
    \sup_{\| A \|_{ \bb B^1}=1 }\left\lvert \int_{\core} \operatorname{curl}\mu^{-1}B\cdot \operatorname{curl}\mu^{-1}A \, \mathrm{d}x  \right\rvert 
    \\
    \leq \sup_{\| A \|_{ \bb B^1}=1 }
    \| B \|_{ H^1(\core) }\| A \|_{ H^1(\core) } 
    \leq \| B \|_{ H^1(\core) }.
\end{multline}
We use control of $u,B$ in $L^2_{t}H^1(\core)$ to deduce
\begin{equation*}
\int_0^T \norm{\cc N(B,u)}_{(\bb B^1)'}+ \norm{\cc D B}_{(\bb B^1)'}\d t <\infty.
\end{equation*}
By Lemma 1.1 in~\cite[Chapter 3]{temam_navierstokes_2001},
this means that $B$ is continuous with values in $(\bb B^1)'$. Therefore it makes sense to test \cref{eq:distribution-time-magnetic} with some $\chi\in \cc D([0,T))$ such that $\chi = 1$ at $t=0$, yielding 
\[
    \int_0^T \langle B, \mu^{-1}A\rangle_{(\bb B^1)',\bb B^1}\chi'(t) \d t - \langle B(0),\mu^{-1}A\rangle_{(\bb B^1)',\bb B^1}
    \\
    =-\int_0^T \langle f^B, A \rangle_{(\bb B^1)',\bb B^1} \chi(t) \d t .
\]
By comparison with \cref{eq:weak-magnetic} (which is satisfied by the weak solution $B$) we get that 
\[\langle B(0),\mu^{-1}A\rangle_{(\bb B^1)',\bb B^1} = \langle B_0,\mu^{-1}A\rangle_{(\bb B^1)',\bb B^1}\]
for every $A\in (\bb B^1)'$, from which we can conclude that $B(0) = B_0$ in $(\bb B^1)'$. Since $B_0\in L^2$, this is also equality in $L^2$. 

Now we detail the proof for $(u,\omega)$. We repackage the velocity equation \cref{eq:weak-momentum} as 
\begin{equation}\label{eq:distribution-time-momentum}
    \frac{\d}{\d t} \langle (u,\omega),( v, \alpha)\rangle_{\bb U^0} = \frac{\d}{\d t} \langle (u,\omega),(v, \alpha)\rangle_{(\bb U^1)',\bb U^1} = \langle f^u, (v, \alpha)\rangle_{(\bb U^1)',\bb U^1}
\end{equation}
for any $(v,\alpha)\in \bb U^1$. We define $f^u$ by 
\begin{multline}
    \langle f^u,(v,\alpha)\rangle_{(\bb U^1)',\bb U^1} = \langle \rho \theta g\widehat{x}-\rho Le_{3}\times u, v\rangle_{L^2(\outcore)} 
    - JL(e_3\times \omega)\cdot\alpha
    \\
    -\langle 2\rho \nu \mathcal{A}^Su -\rho \mathcal{B} u
    + \mu_o^{-1}\mathcal{B} B, v\rangle_{(H^1(\outcore))',H^1(\outcore)}.
\end{multline}
Here, $\cc{B}:H^1(\outcore)\to (H^1(\outcore))'$ is defined by $\langle \cc{B}u,v\rangle = b(u,v,u)$, and $\cc{A}^S:H^1(\outcore) \to (H^1(\outcore))'$ is defined by $\langle \cc{A}^Su,v\rangle  = a^S(u,v)$.
To verify that $f^u\in L^1_t(\bb U^1)'$, we need to make some estimates. We begin with
\begin{multline}\label{eq:momentum-RHS-est}
    \norm{{f}^u}_{(\mathbb U^1)'} = \sup_{\norm{(v,\alpha)}_{\bb U^1} = 1} \abs*{\langle {f}^u ,(v,\alpha)\rangle_{(\bb U^1)',\bb U^1} } 
    \\
    \leq \sup_{\norm{(v,\alpha)}_{\bb U^1} = 1}\p*{\norm{\rho^{1/2}\theta  g\radial}_{L^2(\outcore)} + \norm{\rho^{1/2} Le_3\times u}_{L^2(\outcore)}}\norm{\rho^{1/2}v}_{L^2(\outcore)}+|J^{1/2}Le_3\times \omega |\cdot|J^{1/2}\alpha|
    \\
    +\p*{\norm{2\rho \nu\cc A^S u}_{(H^1(\outcore))'} +\norm{\rho \cc B u}_{(H^1(\outcore))'} + \norm{\mu_o^{-1} \cc B B_m}_{(H^1(\outcore))'}}\norm{\nabla v}_{L^2(\outcore)}
    \\
    \lesssim \norm{\rho^{1/2}\theta  g\radial}_{L^2(\outcore)} + \norm{\rho^{1/2} Le_3\times u}_{L^2(\outcore)}+|J^{1/2}Le_3\times u |
    \\
    +\norm{2\rho \nu\cc A^S u}_{(H^1(\outcore))'} +\norm{\rho \cc B u}_{(H^1(\outcore))'} + \norm{\mu_o^{-1} \cc B B_m}_{(H^1(\outcore))'}.
\end{multline}
The terms $\norm{\rho^{1/2}\theta  g\radial}_{L^2(\outcore)}, \norm{\rho^{1/2} Le_3\times u}_{L^2(\outcore)},|J^{1/2}Le_3\times \omega |$ are all controlled in $L^\infty_t$. For the viscosity term, compute 
\begin{multline}\label{eq:viscosity-est}
    \| 2\rho\nu \cc{A}^Su \|_{ (H^1(\outcore))' } = \sup_{\| v \|_{H^1(\outcore)  }=1 }\left\lvert \int_{\outcore} 2\rho\nu D^Su:D^Sv \, \mathrm{d}x  \right\rvert 
    \\
    \leq \sup_{\| v \|_{H^1(\outcore)}=1 } C\| \nabla u\|_{ L^2(\outcore)}\| \nabla v \|_{ L^2(\outcore) }
    \lesssim 
    \| u \|_{ H^1(\outcore) }  .
\end{multline}
For the convection term, 
\begin{multline}\label{eq:convection-est}
    \|  \cc{B} u \|_{ (H^1(\outcore))' }  =\sup_{\| v \|_{H^1(\outcore)  }=1 }\left\lvert \int_{\outcore} b(u,v,u) \, \mathrm{d}x  \right\rvert 
    \\
    \lesssim \sup_{\| v \|_{H^1(\outcore)  }=1 }\| u \|_{ L^4(\outcore) } ^2\| v \|_{ H^1(\outcore) } 
    \lesssim \| u \|_{ H^1(\outcore) }^2.
\end{multline}
For the Lorentz term, 
\begin{equation}\label{eq:Lorentz-est}
\| \cc{B} B \|_{ (H^1(\outcore))' }=\sup_{\| v\|_{ H^1(\outcore) } =1}\left\lvert \int_{\outcore} b(B,v,B) \, \mathrm{d}x  \right\rvert   
\lesssim \| B \|_{ H^1(\outcore) }^2.
\end{equation}
We use control of $B$ and $u$ in $L^2_{t}H^1_{x}$ to derive 
\[\int_0^T \| 2\rho\nu \cc{A}^Su \|_{ (H^1(\outcore))' }+ \|  \cc{B} u \|_{ (H^1(\outcore))' }+ \norm{\cc B B}_{(H^1(\outcore)) '}\d t<\infty.\]
Putting these estimates together, we conclude
\[ 
\int_0^T  \| {f}^u \|_{(\mathbb{U}^1)'}  \, \mathrm{d}t <\infty.
\]
As before, this implies $(u,\omega)$ is continuous with values in $(\bb U^1)'$. We test \cref{eq:distribution-time-momentum} with some $\chi\in \cc D([0,T))$ such that $\chi = 1$ at $t=0$, and we compare with \cref{eq:weak-momentum} (which is satisfied by the weak solution $(u,\omega$)). We get that 
\[ \langle (u(0),\omega(0)),(v, \alpha)\rangle_{\bb U^0}= \langle (u_0,\omega_0),( v, \alpha)\rangle_{\bb U^0} \]
for every $(v,\alpha)\in (\bb B^1)'$, from which we can conclude that $(u(0),\omega(0)) = (u_0,\omega_0)$ in $(\bb U^1)'$. Since $(u_0,\omega_0)\in \bb U^0$, this is also equality in $\bb U^0$. 

A completely analogous technique works for $\tau$.
\end{proof}

\subsubsection{Limit energy}
Lastly, we verify that the solution of \Cref{lem:existence-weakeqs} satisfies the desired energy inequality.

\begin{lemma}\label{lem:energy-limit}
    Let $(B,(u,\omega),\tau)$ be as in \Cref{lem:existence-weakeqs}. Then they satisfy \cref{eq:energy-total}.
\end{lemma}
\begin{proof}
    We first note that (due to the weak convergences)
    \begin{multline*} 
        \norm{B(t)}_{L^2(\R^3)}, 
        \norm{u(t)}_{L^2(\outcore)}, 
        \norm{\tau(t)}_{L^2(\outcore)}, |\omega(t)|
        \\
        \leq \liminf_m \,
        \norm{B_m(t)}_{L^2(\R^3)}, 
        \norm{u_m(t)}_{L^2(\outcore)}, 
        \norm{\tau_m(t)}_{L^2(\outcore)}, |\omega_m(t)| 
    \end{multline*}
    and 
    \begin{multline*} 
        \norm{\tcurl \mu^{-1} B}_{L^2_tL^2(\core)}, \norm{D^S(u)}_{L^2_tL^2(\outcore)}, \norm{\nabla \tau}_{L^2_tL^2(\outcore)}
        \\
        \leq \liminf_m\, 
        \norm{\tcurl \mu^{-1} B_m}_{L^2_tL^2(\core)},
        \norm{D^S(u_m)}_{L^2_tL^2(\outcore)}, 
        \norm{\nabla \tau_m}_{L^2_tL^2(\outcore)}. 
    \end{multline*}
    We now take the $\liminf$ in $m$ of \cref{eq:galerkin-energy-exact} after integrating in time.
    \begin{multline}\label{eq:energy-limit-initial}
        \frac{1}{2}\p*{
        \norm{\mu^{-1/2} B(t)}_{L^2(\R^3)}^2 + \rho \norm{u(t)}_{L^2(\outcore)}^2 + J |\omega(t)|^2 + \norm{\tau(t)}_{L^2(\outcore)}^2
        } 
        \\
        +\int_0^t\sigma^{-1} \norm{\tcurl \mu^{-1} B}_{L^2(\core)}^2 + 2\rho \nu\norm{\nabla u}_{L^2(\outcore)}^2+\kappa \norm{\nabla \tau}_{L^2(\outcore)}^2\d s
        \\
        \leq \liminf_m \bigg[
        \frac{1}{2}\p*{
        \norm{\mu^{-1/2} B_m(t)}_{L^2(\R^3)}^2 + \rho \norm{u_m(t)}_{L^2(\outcore)}^2 + J |\omega_m(t)|^2 + \norm{\tau_m(t)}_{L^2(\outcore)}^2
        } 
        \\
        +\int_0^t\sigma^{-1} \norm{\tcurl \mu^{-1} B_m}_{L^2(\core)}^2 + 2\rho \nu\norm{\nabla u_m}_{L^2(\outcore)}^2+\kappa \norm{\nabla \tau_m}_{L^2(\outcore)}^2\d s \bigg]
        \\
        =
        \liminf_m \bigg[
        \frac{1}{2}\p*{
        \norm{\mu^{-1/2} B_m(0)}_{L^2(\R^3)}^2 + \rho\norm{u_m(0)}_{L^2(\outcore)}^2 + J |\omega_m(0)|^2 + \norm{\tau_m(0)}_{L^2(\outcore)}^2
        }
        \\
        +\int_0^t \langle \theta^c g\radial,\rho u_m\rangle_\outcore +\langle \tau_m g\radial , \rho u_m\rangle_\outcore -b(u_m,\theta^c,\tau_m)-\kappa a(\theta^c,\tau_m)-\kappa \langle f_b,\tau_m\rangle_{\partial \incore}+\langle f, \tau_m\rangle_\outcore \d t\bigg].
    \end{multline}
    Recall that 
    \[B_m(0), (u,\omega)_m(0), \tau_m(0) \xrightarrow[L^2_x]{} B_0, (u,\omega)_0, \tau_0\]
    strongly because the initial conditions were defined by orthogonal projection in $L^2$. Thus, the $\liminf$ can be replaced by the limit for those terms. We can use weak convergence to replace many of the remaining terms with their limits; in particular, 
    \begin{align*}
        \lim_{m\to\infty} a(\theta^c,\tau_m)& = a(\theta^c,\tau),
        &\quad 
        \lim_{m\to\infty}\langle \theta^cg\widehat{x},\rho u_m\rangle_\outcore &=\langle \theta^c g\widehat{x},\rho u\rangle_\outcore,
        \\
        \lim_{m\to\infty}\langle f, \tau_m\rangle_{(\bb T^1)', \bb T^1} &= \langle f,\tau\rangle_{(\bb T^1)', \bb T^1},
        &\quad 
        \lim_{m\to\infty}\langle f_b,\tau_m\rangle_{H^{-1/2}(\partial\incore),H^{1/2}(\partial \incore)} &= \langle f_b,\tau\rangle_{H^{-1/2}(\partial\incore),H^{1/2}(\partial\incore)},
    \end{align*}
    by the weak convergence of $\tau_m$ and $u_m$ in $L^2_tH^1(\outcore)$.
    This leaves us with two terms to analyze: $\langle \tau_mg\radial ,\rho u_m\rangle_\outcore$ and $b(u_m,\theta^c,\tau_m)$. Here, we will need to leverage strong convergence of $u_m$ in $L^2_tL^2(\outcore)$. We write
    \[
    \int_0^t \langle \tau_m g\widehat{x}, \rho u_m\rangle_\outcore \d t =\int_0^t \langle \tau_m g\widehat{x}, \rho u\rangle_\outcore \d +\int_0^t \langle\tau_m g\widehat{x}, \rho u_m - \rho u \rangle_\outcore \d t
    \]
    and observe that the first term converges to $\int_0^t \langle \tau g\widehat{x},\rho  u\rangle\d t $ due to weak convergence of $\tau_m$ in $L^2_tH^1(\outcore)$. The second term converges to zero because of strong convergence of $u_m$ in $L^2_tL^2(\outcore)$ (and uniform boundedness of $\tau_m$ in $L^2_tH^1(\outcore)$). Similarly,
    \[
    \int_0^t b(u_m, \theta^c,\tau_m) \d t =\int_0^t b(u, \theta^c,\tau_m)\d t+\int_0^t  b(u_m-u, \theta^c,\tau_m) \d t.
    \]
    Again, the first term converges to $\int_0^t b(u, \theta^c, \tau)\d t$ due to weak convergence of $\tau_m$ in $L^2_tH^1(\outcore)$. The second term converges to zero using Ladyzhenskaya's inequality for $u$, strong convergence in $L^2_tL^2(\outcore)$, and boundedness of $\tau_m$ in $L^2_tH^1(\outcore)$ (c.f. \cref{eq:nonlin-ladyzhenskaya}).
    
    We can now pass to the limit in all the terms of \cref{eq:energy-limit-initial} to deduce 
    \begin{multline*}
        \frac{1}{2}\p*{
        \norm{\mu^{-1/2}B(t)}_{L^2(\R^3)}^2 + \rho\norm{u(t)}_{L^2(\outcore)}^2 + J|\omega(t)|^2 + \norm{\tau(t)}_{L^2(\outcore)}^2
        } 
        \\
        +\int_0^t\sigma^{-1} \norm{\tcurl \mu^{-1} B}_{L^2(\core)}^2 + 2\rho \nu\norm{\nabla u}_{L^2(\outcore)}^2+\kappa \norm{\nabla \tau}_{L^2(\outcore)}^2\d s
        \\
        \leq 
        \lim_m
        \frac{1}{2}\p*{
        \norm{\mu^{-1/2}B_m(0)}_{L^2(\R^3)}^2 + \rho \norm{u_m(0)}_{L^2(\outcore)}^2 + J |\omega_m(0)|^2 + \norm{\tau_m(0)}_{L^2(\outcore)}^2
        }
        \\
        +\int_0^t \langle \theta^c g\radial,\rho u_m\rangle_\outcore +\langle \tau_m g\radial , \rho u_m\rangle_\outcore  - b(u_m,\theta^c,\tau_m) - \kappa a(\theta^c,\tau_m)
        \\
        -\kappa \langle f_b,\tau_m\rangle_{H^{-1/2}(\partial\incore),H^{1/2}(\partial \incore)}+\langle f, \tau_m\rangle_{(\bb T^1)',\bb T} \d t
        \\
        =
        \frac{1}{2}\p*{
        \norm{\mu^{-1}B(0)}_{L^2(\R^3)}^2 + \rho\norm{u(0)}_{L^2(\outcore)}^2 + J |\omega(0)|^2 + \norm{\tau(0)}_{L^2(\outcore)}^2
        }
        \\
        +\int_0^t \langle \theta^c g\radial,\rho u\rangle_\outcore +\langle \tau g\radial ,\rho  u\rangle_\outcore  -b(u,\theta^c,\tau)-\kappa a(\theta^c,\tau)
        \\
        -\kappa \langle f_b,\tau\rangle_{H^{-1/2}(\partial \incore),H^{1/2}(\partial\incore)}+\langle f, \tau\rangle_{(\bb T^1)',\bb T} \d t
    \end{multline*}
    which is exactly \cref{eq:energy-total}.
\end{proof}

\subsubsection{Weak continuity}
\begin{lemma}\label{lem:weak-continuity}
    Let $(B,(u,\omega),\tau)$ be as in \Cref{lem:existence-weakeqs}. Then $(B,(u,\omega),\tau)$ is weakly continuous with values in $\bb{B}^0\times \bb{U}^0\times \bb{T}^0$.
\end{lemma}
\begin{proof}
    In the proof of \Cref{lem:IC-limit}, it was observed that $(B,(u,\omega),\tau)\in \cc{C}^0([0,T];(\bb{B}^1)'\times (\bb{U}^1)'\times (\bb{T}^1)')$. We also have that $(B,(u,\omega),\tau)\in L^\infty([0,T];\bb{B}^0\times \bb{U}^0\times \bb{T}^0)$. It is standard (for example, Lemma II.5.9 in~\cite{boyer_mathematical_2013}) to conclude that $(B,(u,\omega),\tau)\in \mathcal{C}^0_{\mathrm{weak}}([0,T];\bb{B}^0\times \bb{U}^0\times \bb{T}^0)$
\end{proof}

\subsubsection{Complete proof}
Finally, we can prove \Cref{thm:existence}.
\begin{proof}[Proof of \Cref{thm:existence}]
    Apply \Cref{lem:existence-weakeqs,lem:IC-limit,lem:energy-limit,lem:weak-continuity}.
\end{proof}

\appendix
\section{Technical results}\label{sec:appendix}
\subsection{Compactness}\label{sec:appendix-compactness}
In this section, we will think of the function $f$ as a map from time into a Hilbert space, rather than a map from space and time.

Let $X$ and $Y$ be Hilbert spaces. Define 
\[ 
H^\gamma(\mathbb{R};X,Y)=\{ f\in L^2(\mathbb{R};X): D_{t}^\gamma f\in L^2(\mathbb{R};Y) \},
\]
and 
\[ 
H^\gamma_{T}(\mathbb{R};X,Y)=\{ f\in H^\gamma(\mathbb{R};X,Y):\operatorname{supp}f\subseteq [0,T] \}.
\]
Here, $D_t^\gamma$
is the Fourier multiplier corresponding to multiplication by $|\hat t|^\gamma$ on the Fourier side.
\begin{theorem}
    \label{thm:fractional-ALS}
    Assume $X_0\subset X \subset X_1$ are three Hilbert spaces and the first embedding is compact. Then for finite $T$ and any $\gamma>0$ the embedding of $H_{T}^\gamma(\mathbb{R};X_0,X_1)$ into $L^2(\mathbb{R},X)$ is compact.
\end{theorem}
This is Theorem 2.2 in~\cite[Chapter 3]{temam_navierstokes_2001}.

The following lemma is a useful condition to check that a function $f\in H^\gamma_T(\bb{R};X, Y)$. The proof can be found in~\cite[pages 193-194]{temam_navierstokes_2001}.
\begin{lemma}
    \label{lem:fractional-deriv}
    Suppose $|\hat t|\norm{\widehat{f}(\hat t)}_{X}^2\leq C \norm{\widehat{f}(\hat t)}_{Y}$ where $\widehat{f}$ is the Fourier transform in time of $f$ and $f\in L^2_tY$. Then, for any $\gamma<1/4$,
    \[
    \int_\R |\hat t|^{2\gamma}\norm{\widehat{f}(\hat t)}^2_{X}\d \hat t 
    \lesssim
    \p*{\int_\R \frac{1}{(1+|\hat t|^{1-2\gamma})^2}\d \hat t}^{1/2}\p*{\int_\R \norm{f(t)}_{Y}^2\d t}^{1/2}
    +
    \int_{\bb R} \norm{\widehat{f}(\hat t)}_Y^2 \d \hat t
    \leq 
    C.
    \]
\end{lemma}
With this background, we can now prove \Cref{lem:compactness}.
We restate it for convenience: 
\getkeytheorem{lem:compactness}
\begin{proof}[Proof of \Cref{lem:compactness}]
    We will closely follow Chapter 3, Section 3.2 in \cite{temam_navierstokes_2001}.
    We already know that $\{ (u,\omega)_{m} \}$ is uniformly bounded in $L^2_{t}\mathbb{U}^1$. By considering the extension $u$ to $\incore$ (according to \cref{eq:momentum-inner-core}) and using 
    \[\norm{u}_{H^1_{0}(\core)} \lesssim \norm{(u,\omega)}_{\bb U ^1}\]
    (which is an intermediate result in \Cref{lem:poincare-korn})
    we conclude that the extended $u$ is uniformly bounded in $L^2_tH^1_0(\core)$.
    
    Next we will show that for the extended $u$, $D_t^\gamma u\in L^2_tL^2(\outcore)$.
    We begin by extending $(u,\omega)_m$ to be zero outside of $[0,T]$ (note this extension is not necessarily continuous). We then test these extended functions against $\rho (v, \alpha)_j$ using \cref{eq:momentum-and-angular-galerkin}.
    \[ 
    \frac{d}{dt} \langle (u,\omega)_m,(v,\alpha)_j\rangle _{\bb U^0}=\left\langle  \widetilde{f}^u_{m} , ( v,\alpha)_{j} \right\rangle _{(\bb U^1)',\mathbb{U}^1}+\left\langle ({u},{\omega})_{m}(0) , (v,\alpha)_{j} \right\rangle_{\bb U^0}\delta_{0}-\left\langle ({u},{\omega})_{m}(T) , (v,\alpha)_j \right\rangle_{\bb U^0} \delta_{T} .
    \]
    where $\left\langle  \widetilde{f}^u_{m} , ( v,\alpha)_{j} \right\rangle _{(\bb U^1)',\mathbb{U}^1}$ is defined by 
    \[ 
    \begin{cases}
    \begin{aligned}
        \langle \rho &\theta_m g\widehat{x}-\rho Le_{3}\times u_m, v_j\rangle_{L^2(\outcore)} 
        - JL(e_3\times \omega_{m})\cdot\alpha_j
        \\
        &-\langle 2\rho \nu \mathcal{A}^Su_{m} -\rho \mathcal{B} u_{m}
        + \mu_o^{-1}\mathcal{B} B_{m}, v_j\rangle_{(H^1(\outcore))',H^1(\outcore)} 
    \end{aligned} & t\in[0,T] \\
    0 & \mathrm{otherwise}.
    \end{cases}
    \]
    Now we take the Fourier transform in time.
    \begin{multline*}
        2\pi i\hat t \left\langle \widehat{(u,\omega)}_{m} , (v,\alpha)_j \right\rangle_{\bb U ^0} =\left\langle \widehat{f}^u_{m} , (v,\alpha)_j \right\rangle _{(\mathbb{U}^1)',\mathbb{U}^1}
        \\
        +\left\langle ({u},{\omega})_{m}(0) , \rho(v,\alpha)_j \right\rangle_{\bb U^0}-\left\langle ({u},{\omega})_{m}(T) , \rho(v,\alpha)_j \right\rangle_{\bb U^0} e^{ 2\pi iT\hat t }.
    \end{multline*}
    We can now effectively test against $\widehat{(u,\omega)}_{m}$ by summing the above for $j=1,\dots,n$ after multiplying each by $\widehat{g}_{j,m}^u$.
    \begin{multline*}
        2\pi i\hat t \left\langle \widehat{(u,\omega)}_{m} , \widehat{(u,\omega)}_{m} \right\rangle_{\bb U^0} =\left\langle \widehat{f}^u_{m} , \widehat{(u,\omega)}_{m} \right\rangle _{(\mathbb{U}^1)',\mathbb{U}^1}
        \\
        +\left\langle ({u},{\omega})_{m}(0) ,  \widehat{(u,\omega)}_{m} \right\rangle_{\bb U^0}-\left\langle ({u},{\omega})_{m}(T) , \widehat{(u,\omega)}_{m} \right\rangle_{\bb U^0} e^{ 2\pi iT\hat t }.
    \end{multline*}
    It follows from estimates \cref{eq:momentum-RHS-est,eq:viscosity-est,eq:convection-est,eq:Lorentz-est}, control of $u_{m},B_{m}$ in $L^2_{t}H^1(\core)$ uniformly in $m$, and control of $\norm{\theta_m}_{L^2(\outcore)},\norm{u_m}_{L^2(\outcore)},|\omega_m|$ in $L^\infty_t$ uniformly in $m$, that
    \[ 
    \int_\R  \| \widetilde{f}^u_{m} \|_{(\mathbb{U}^1)'}  \, \mathrm{d}t <\infty 
    \]
    also uniformly in $m$. This implies that
    \[ 
    \sup_{\hat t \in \mathbb{R}} \| \widehat{f}^u_{m}(\hat t) \|_{(\mathbb{U}^1)'} <\infty
    \]
    uniformly in $m$. We also know (using control of $u$ in $L^\infty_{t}L^2_{x}$ and $\omega$ in $L^\infty_{t}\mathbb{R}^3$) that for almost every $T$
    \[ 
    \| (u,\omega )_{m}(0) \|_{ \bb U ^0} ,\ \| (u,\omega )_{m}(T) \|_{ \bb U ^0 }<\infty
    \]
    also uniformly in $m$. Thus
    \begin{multline*}
    \left\lvert \hat t  \right\rvert   \| \widehat{(u,\omega)}_{m}(\hat t) \|_{ \bb U ^0} ^2
    \\
    \leq \left( \| \widehat{f}^u_{m}(\hat t) \|_{ \mathbb{U}' }+\| (u,\omega )_{m}(0) \|_{ \bb U ^0 } +\| (u,\omega )_{m}(T) \|_{ \bb U ^0 } \right)  \| \widehat{ (u,\omega)}_{m}(\hat t) \|_{  \mathbb{U}^1} 
    \\
    \lesssim \| \widehat{(u,\omega)}_{m}(\hat t) \|_{  \mathbb{U}^1}.
    \end{multline*}
    We observe (see \cref{eq:u-L2-norm-equiv}),
    \[
    \norm{\widehat{u}_m(\hat t)}_{L^2(\core)}\lesssim \norm{\widehat{(u,\omega)}_m(\hat t)}_{\bb U^0}
    \]
    (where $u$ on the left hand side is extended to $\incore$ according to \cref{eq:momentum-inner-core}),
    so we can apply \Cref{lem:fractional-deriv} to conclude that the sequence $\{{u}_{m}\}$ is bounded in $H_T^\gamma(\mathbb{R},H^1_0(\core), L^2(\core))$
    for any fixed $\gamma < 1/4$. Apply \Cref{thm:fractional-ALS} (take $X_1 = X= L^2(\core), X_0 = H^1(\core)$)
    to conclude that there is a subsequence (we do not change the index for convenience) that converges in the strong topology:
    \[ 
    u_m \xrightarrow[L^2_{t}L^2(\core)]{} u.
    \]
    The analysis for $B$ is similar. We only need to show strong convergence in $L^2_tL^2(\core)$ (rather than $L^2_tL^2(\R^3)$)
    because all relevant integrals in the weak formulation are over $\core$ or $\outcore$.
    Firstly, we recall
    \[\norm{B}_{H^1(\incore)\times H^1(\outcore)}\lesssim\norm{B}_{\bb B^1}\]
    so that $B_m$ are uniformly bounded in $L^2_t(H^1(\incore)\times H^1(\outcore))$.
    Now, following the same analysis as for the momentum, we arrive at the following inequality
    \begin{multline*}
    \left\lvert \hat t \right\rvert \| \mu^{-1/2}\widehat{B}_{m}(\hat t) \|_{ L^2(\mathbb{R}^3)  } ^2
    \\
    \leq \left( \| \widehat{f}^B_{m}(\hat t) \|
    _{(\bb B^1)'}
    +\|\mu^{-1/2} B_{m}(0) \|_{ L^2(\bb R^3)} +\| \mu^{-1/2}B_{m}(T) \|_{ L^2(\R^3)} \right)  \| \mu^{-1/2}\widehat{B}_{m}(\hat t) \|
    _{\bb B^1}
    \\
    \lesssim \|\mu^{-1/2} \widehat{B}_{m}(\hat t) \|
    _{\bb B^1}.
    \end{multline*}
    Here, $\widetilde{f}^B_m$ is analogous to but different from $\widetilde{f}^u_m$ from the momentum, i.e., 
    \[
    \left\langle \widetilde{f}^B_m, A_j\right\rangle _{(\bb B^1)',\bb B^1} = 
    \begin{cases}
        \langle \cc{N}(u_m,B_m)-\cc{D}B_m, A_j\rangle_{(\bb B^1)',\bb B^1}  & t\in[0,T]
        \\
        0 & \text{otherwise}\,.
    \end{cases}
    \]
    As before, we must demonstrate that 
    \begin{equation}\label{eq:fB-finite-integral}
        \int_\R \norm{\widetilde{f}^B_m}_{(\bb B^1)'} \d t<\infty .
    \end{equation}
    It follows from estimates \cref{eq:frozen-in-flux-est,eq:magnetic-diffusion-est} and control of $u_{m},B_{m}$ in $L^2_{t}H^1(\core)$ uniformly in $m$ that
    \[
    \int_0^T \norm{\cc N(B_m,u_m)}_{(\bb B^1)'}+ \norm{\cc D B_m}_{(\bb B^1)'}\d t <\infty,
    \]
    uniformly in $m$, and therefore
    \[ 
    \sup_{\hat t \in \mathbb{R}} \| \widehat{f}^B_{m}(\hat t) \|_{(\mathbb{U}^1)'} <\infty.
    \]
    As before, we deduce that (for almost every $T$)
    \[
    \left\lvert \hat t \right\rvert \| \widehat{B}_{m}(\hat t) \|_{ L^2(\core)  } ^2
    \leq
    \left\lvert \hat t \right\rvert \| \widehat{B}_{m}(\hat t) \|_{ L^2(\mathbb{R}^3)  } ^2
    \lesssim 
    \| \widehat{B}_{m}(\hat t) \|_{ \bb B^1}.
    \]
    Therefore $B_m$ is bounded in $H^\gamma_T(\R,H^1_{\tdiv}(\incore)\times H^1_{\tdiv}(\outcore),L^2(\core))$ for fixed $\gamma < 1/4$ and we use \Cref{thm:fractional-ALS} to conclude (up to a subsequence) the strong convergence
    \[B_m\xrightarrow[L^2_tL^2(\core)]{} B.\]
\end{proof}

\subsection{Momentum function space}
We collect two explicit computations.
Let $(u,\omega)\in \bb U^1$ and consider $u$ extended to $\incore$ according to \cref{eq:momentum-inner-core}. Then, on $\incore$,
\begin{equation}\label{eq:grad-u-incore}
    \nabla u = \begin{bmatrix}
    0 & -\omega_{3} & \omega_{2} \\
    \omega_{3} & 0 & -\omega_{1} \\
    -\omega_{2} & \omega_{1} & 0 \\
    \end{bmatrix}.
\end{equation}
Introducing $\rho_i$ as the density of the inner core, we also have
\begin{equation}\label{eq:u-L2-norm-equiv}
    \norm{(\rho^{1/2}\bbm 1_\outcore + \rho_i^{1/2}\bbm 1_\incore) u}_{L^2(\core)}^2 = \int_\outcore \rho u^2\d x+\int_\incore \rho_i (\omega\times x)^2\d x =\norm{(u,\omega)}_{\bb U^0}^2
\end{equation}
where, recalling that $J$ is the moment of inertia of the inner core, the second integral is equivalent to $J|\omega|^2$.

\bibliographystyle{amsplain}
\bibliography{references}

\end{document}